\documentclass[pdflatex,sn-mathphys-num]{sn-jnl}% Math and Physical Sciences Numbered Reference Style
\usepackage{graphicx}%
\usepackage{multirow}%
\usepackage{amsmath,amssymb,amsfonts}%
\usepackage{amsthm}%
\usepackage{mathrsfs}%
\usepackage[title]{appendix}%
\usepackage{xcolor}%
\usepackage{textcomp}%
\usepackage{manyfoot}%
\usepackage{booktabs}%
\usepackage{algorithm}%
\usepackage{algorithmicx}%
\usepackage{algpseudocode}%
\usepackage{listings}%
\usepackage{comment}
\usepackage{mathtools}
\usepackage[normalem]{ulem} % to use \sout

\newcommand{\cut}[1]{{}}

\newcommand{\xu}[1]{{\color{black}#1}}% comments
\newcommand{\vh}{{\mathbf{h}}}

\newcommand{\vr}{{\mathbf{r}}}

\newcommand{\vy}{{\mathbf{y}}}

\newcommand{\vA}{{\mathbf{A}}}

\newcommand{\vD}{{\mathbf{D}}}

\newcommand{\vG}{{\mathbf{G}}}
\newcommand{\vH}{{\mathbf{H}}}
\newcommand{\vI}{{\mathbf{I}}}

\newcommand{\vM}{{\mathbf{M}}}

\newcommand{\vP}{{\mathbf{P}}}

\newcommand{\vV}{{\mathbf{V}}}
\newcommand{\vW}{{\mathbf{W}}}
\newcommand{\vX}{{\mathbf{X}}}

\newcommand{\vZ}{{\mathbf{Z}}}

\newcommand{\cE}{{\mathcal{E}}}

\newcommand{\cG}{{\mathcal{G}}}

\newcommand{\cL}{{\mathcal{L}}}

\newcommand{\cN}{{\mathcal{N}}}

\newcommand{\cU}{{\mathcal{U}}}
\newcommand{\cV}{{\mathcal{V}}}

\newcommand{\EE}{{\mathbb{E}}}

\newcommand{\bm}[1]{\boldsymbol{#1}}

\newcommand{\RR}{\mathbb{R}}

\newcommand{\vone}{{\mathbf{1}}}

\newcommand{\bc}{\begin{center}}
\newcommand{\ec}{\end{center}}

\newcommand{\bdm}{\begin{displaymath}}
\newcommand{\edm}{\end{displaymath}}

\newcommand{\beq}{\begin{equation}}
\newcommand{\eeq}{\end{equation}}

\newcommand{\bfl}{\begin{flushleft}}
\newcommand{\efl}{\end{flushleft}}

\newcommand{\et}{\end{tabbing}}

\newcommand{\beqn}{\begin{align}}
\newcommand{\eeqn}{\end{align}}

\newcommand{\beqs}{\begin{align*}} % no equation numbers
\newcommand{\eeqs}{\end{align*}}  % no equation numbers

\newcommand{\svi}{\sqrt{\vV_{i}}}
\newcommand{\svip}{\sqrt{\vV_{i-1}}}
\newcommand{\sv}{\sqrt{\vV_{t}}}
\newcommand{\svp}{\sqrt{\vV_{t-1}}}

\newcommand{\mt}{\vM_{t}}
\newcommand{\mtp}{\vM_{t-1}}

\newcommand{\mip}{\vM_{i-1}}
\newcommand{\gt}{\vG_{t}}

\newcommand{\at}{\alpha}
\newcommand{\atp}{\alpha}
\newcommand{\wt}{\vW_{t}}
\newcommand{\wtp}{\vW_{t-1}}
\newcommand{\wtn}{\vW_{t+1}}
\newcommand{\xt}{\widetilde{\vW}_t}
\newcommand{\xtn}{\widetilde{\vW}_{t+1}}

\newcommand{\gi}{\vG_{i}}

\newcommand{\ai}{\alpha}
\newcommand{\aip}{\alpha}

\newcommand{\xix}{\widetilde{\vW}_i}
\newcommand{\xin}{\widetilde{\vW}_{i+1}}

\newcommand{\sumi}{\sum_{i=1}^t}
\newcommand{\gfx}{\nabla F(\widetilde{\vW}_i)}

\newcommand{\bfracConst}{\frac{\beta_1}{1-\beta_1}}
\theoremstyle{thmstyleone}%
\newtheorem{theorem}{Theorem}%  meant for continuous numbers
\newtheorem{assumption}{Assumption}

\newtheorem{lemma}{Lemma}

\theoremstyle{thmstyletwo}%
\newtheorem{remark}{Remark}%

\theoremstyle{thmstylethree}%

\begin{document}

\title[Article Title]{Neighbor-Sampling Based Momentum Stochastic Methods for Training Graph Neural Networks}

%%=============================================================%%
%% GivenName	-> \fnm{Joergen W.}
%% Particle	-> \spfx{van der} -> surname prefix
%% FamilyName	-> \sur{Ploeg}
%% Suffix	-> \sfx{IV}
%% \author*[1,2]{\fnm{Joergen W.} \spfx{van der} \sur{Ploeg} 
%%  \sfx{IV}}\email{iauthor@gmail.com}
%%=============================================================%%

\author[1]{\fnm{Molly} \sur{Noel}}\email{noelm@rpi.edu}

\author[1]{\fnm{Gabriel} \sur{Mancino-Ball}}\email{gabriel.mancino.ball@gmail.com}
%\equalcont{These authors contributed equally to this work.}

\author*[1]{\fnm{Yangyang} \sur{Xu}}\email{xuy21@rpi.edu}
%\equalcont{These authors contributed equally to this work.}

\affil[1]{\orgdiv{Mathematical Sciences}, \orgname{Rensselaer Polytechnic Institute}, \orgaddress{\street{110 Eighth Street}, \city{Troy}, \postcode{12180}, \state{New York}, \country{United States}}}

%\affil[2]{\orgdiv{Department}, \orgname{Organization}, \orgaddress{\street{Street}, \city{City}, \postcode{10587}, \state{State}, \country{Country}}}

%\affil[3]{\orgdiv{Department}, \orgname{Organization}, \orgaddress{\street{Street}, \city{City}, \postcode{610101}, \state{State}, \country{Country}}}

%%==================================%%
%% Sample for unstructured abstract %%
%%==================================%%

\abstract{
 Graph convolutional networks (GCNs) are a powerful tool for graph representation learning. Due to the recursive neighborhood aggregations employed by GCNs, efficient training methods suffer from a lack of theoretical guarantees or are missing important practical elements from modern deep learning algorithms, such as adaptivity and momentum. In this paper, we present several neighbor-sampling (NS) based Adam-type stochastic methods for solving a nonconvex GCN training problem. We utilize the control variate technique proposed by \cite{VRGCN} to reduce the stochastic error caused by neighbor sampling. Under standard assumptions for Adam-type methods, we show that our methods enjoy the optimal convergence rate. In addition, we conduct extensive numerical experiments on node classification tasks with several benchmark datasets. The results demonstrate superior performance of our methods over classic NS-based SGD that also uses the control-variate technique, especially for large-scale graph datasets. Our code is available at \url{https://github.com/RPI-OPT/CV-ADAM-GNN}. %https://github.com/molly-noel/CV-Adam-GNN.
}


\maketitle

\section{Introduction}\label{sec1}
Graph structured data is ubiquitous in our world. To leverage the benefits of recent advancements in deep learning while exploiting available graph structure, graph representation learning has emerged as a general framework for tackling many graph-based tasks~\cite{chen2020graph}. Graph representation learning has found success in a wide range of applications including recommendation systems~\cite{fan2019graph,sharma2024survey}, weather forecasting~\cite{lam2023learning}, quantum chemistry~\cite{gilmer2017neural}, and code clone detection~\cite{liu2023learning}.

Graph Neural Networks (GNNs), first introduced by \cite{gori2005new,scarselli2008graph}, comprise a framework of graph representation learning methods that recursively aggregate local node information with that of their neighbors. Graph Convolutional Networks (GCNs) are a popular GNN first introduced by \cite{OriginalGCN} which
have sparked many advancements in graph representation learning including architecture design \cite{li2020deepergcn,wolfe2023gist} and training considerations \cite{NS_GraphSAGE,chen2018fastgcn,zeng2019graphsaint,li2021training}. This paper focuses on the training dynamics of GCNs from an optimization perspective; more precisely, we develop algorithms that incorporate the control variate estimator (CVE)~\cite{VRGCN} into several Adam-type optimizers for GCN training.

Since GCNs aggregate neighbor information recursively, classical training algorithms, such as stochastic gradient descent (SGD)~\cite{SGDrobbins1951stochastic} or its various momentum-based variants~\cite{Nesterov2014IntroductoryLO}, can be computationally expensive for large and densely connected graphs. %\cite{VRGCN} addresses this issue by proposing a CVE technique and using the CVE-based sample gradient within the classic SGD. 
By maintaining historical approximations of node features at each layer, the CVE-based SGD introduced in~\cite{VRGCN} is able to significantly decrease the amount of recursive node representation computations. A convergence guarantee to a stationary point is established by \cite{VRGCN} for the CVE-based SGD. However, it is unknown whether Adam-type methods, which promise faster empirical convergence in deep learning on regular (e.g., image/video) data \cite{OriginalAdam, xu2023parallel}, %\gmb{Put a citation here}, 
equipped with the CVE technique can also have guaranteed convergence. The numerical experiments in~\cite{VRGCN} demonstrate the performance of a CVE-based Adam method for training GCNs, but no theoretical performance guarantees for such method have been established. 

\subsection{Contributions}

In this work, we explore the effects of using different Adam-type optimizers together with the CVE technique for training GCNs. The optimizers that we explore incorporate momentum into the gradient and/or effective stepsize components, which are designed to accelerate convergence of stochastic gradient-type methods. We focus on four momentum-based optimizers: Adam \cite{OriginalAdam}, Heavy-Ball SGD \cite{Heavyballpolyak1964some},  AMSGrad \cite{AMSGradreddi2019convergence}, and AdaGrad \cite{duchi2011adaptive}. We not only compare their empirical performance to that of the classic SGD but also provide rigorous convergence analysis to establish optimal convergence rate results.

The work \cite{AdamAnalysis} provides convergence analysis for generalized Adam-type optimizers, which can be specified to include the four aforementioned optimizers, as well as SGD, for a nonconvex stochastic optimization problem. However, its convergence results do not apply to the CVE-based Adam-type methods because their theoretical analysis requires \emph{unbiased} gradients, but the gradients produced from the CVE-based methods are \emph{biased}. In contrast, we provide a convergence guarantee for the general CVE-based Adam-type update case without the unbiased gradient assumption, and optimal convergence rate results for AMSGrad, Heavy-Ball SGD, SGD, and AdaGrad on training GCNs. While our theorem utilizes the specialty of the CVE-based stochastic gradient of a GCN model, it directly applies to other applications with an access to biased gradients where the bias is in the order of the stepsize. We also demonstrate the performance of the five different optimizers for training GCNs on five benchmark node classification datasets.

\subsection{Notation}
Multiplication, division, and square roots of sequences of matrices are performed componentwise, with multiplication denoted by $\odot$. %Division and square roots of sequences of matrices is also componentwise. 
The norm $\|\cdot \|$ is the Frobenius norm unless otherwise stated. The component of a matrix $\vA$ at the $i$-th row and $j$-th column is denoted as $\vA[i,j]$. The trainable model parameter $\vW$ is in the format of a set of matrices, and $(\vW)_j$ represents the $j$-th entry of $\vW$ by viewing it in a long-vector format.

\subsection{Outline}
The rest of the paper is organized as follows. In Section~\ref{sec:gcn}, we give the description and formulation of the GCN training problem, as well as our algorithm. %Important techniques that will be used in our algorithm are introduced there. 
In Section~\ref{sec:relatedwork}, we review existing works about GNNs and algorithms for training GNNs. %Our algorithm and its 
Convergence results of our algorithm are presented in Section %s~\ref{sec:alg} and 
\ref{sec:convergence}. In Section~\ref{sec:numerical}, we show experimental results to demonstrate the effectiveness of our algorithm for training GCNs on a few benchmark datasets. Conclusions are given in Section~\ref{sec:conclusion}.

\section{Formulation of GCN Training and Proposed Algorithm}\label{sec:gcn}

We define an undirected graph $\cG=(\cV,\cE)$ by a set of nodes $\cV=\{1,\ldots,n\}$ and edges $\cE\subseteq\cV\times\cV$. We define the set of neighbors of node $v$ as $\cN_v=\{u:(u,v)\in\cE\}$. The feature matrix of $\cG$ is given by $\vX\in \RR^{n\times d_0}$ such that the $v$-th row of $\vX$ holds the feature vector of node $v$. We denote the adjacency matrix of $\cG$ as $\vA\in \RR^{n\times n}$ such that $\vA[u,v]=1$ if $(u,v)\in\cE$ and 0 otherwise. We now define the GCN architecture and state our problem of interest.

We leverage the normalized adjacency matrix~\cite{OriginalGCN} $\vP= \widetilde{\vD}^{-\frac{1}{2}}(\vA+\vI)\widetilde{\vD}^{-\frac{1}{2}}\in \RR^{n \times n}$, where $\vI$ is the $n \times n$ identity matrix and $\widetilde{\vD}$ is the diagonal matrix such that $\widetilde{\vD}[u,u]=\sum_v (\vA+\vI)[u,v]$. We define the feature representation matrix at layer $k$ of a GCN to be $\vH_{\mathrm{exact}}^{(k)}\in \RR^{n\times d_k}$ such that the $v$-th row corresponds to the representation of node $v$ at layer $k$. We set $\vH_{\mathrm{exact}}^{(0)}=\vX$ and utilize the following update rules:
\begin{align}
    \Tilde{\vZ}^{(k)}&=\vP\vH^{(k-1)}_{\mathrm{exact}}\vW^{(k-1)},\;\; k=1,\ldots,K, \;\label{eq:GCN1}\\    \vH^{(k)}_{\mathrm{exact}}&=\sigma(\Tilde{\vZ}^{(k)}),\;\; k=1,\ldots,K, \label{eq:GCN2}
\end{align}
where $\vW^{(k)}\in \RR^{d_k \times d_{k+1}}$ is a trainable weight matrix and $\sigma$ is a non-linear activation function (e.g. ReLU) at layers $k=1,\dots,K-2$ and an appropriate readout function (e.g. softmax) at layer $K-1$. Updates~\eqref{eq:GCN1} and~\eqref{eq:GCN2} comprise the GCN architecture~\cite{OriginalGCN}. 

We consider the optimization problem of training a $K$-layer GCN for a node-level classification task. The final node embeddings generated by our method can also be used for the edge-classification and graph-classification tasks. Denote $\cV_{\cL}$ as the set of nodes which have a known class label and $\cV_{\cU}=\cV \backslash \cV_{\cL}$ as the set of nodes whose label we would like to predict. %which we would like to perform prediction. 
Let $\vW= \{\vW^{(0)},\vW^{(1)},\ldots,\vW^{(K-1)}\}$ be a sequence of trainable weight matrices such that $\vW^{(k)}$ is the weight matrix at layer $k$ of a GCN. 

The goal is to find weight matrices $\vW$ that minimize the following training loss function $F(\vW)$:
\begin{align}
  F^*:=  \min_{\vW}F(\vW):=\frac{1}{|\cV_{\cL}|}\sum_{v\in \cV_{\cL}}f(\vy_v,{\vh}^{(K)}_v) \label{eq:loss}
\end{align}
where $f$ is an appropriate loss function, $\vy_v$ is the true label for node $v\in\cV_{\cL}$, and ${\vh}^{(K)}_v={\vH}_{\mathrm{exact}}^{(K)}[v,:]^{\top}$.

\subsection{Neighbor Sampling and Receptive Fields}
As shown in \eqref{eq:GCN1}, each feature vector ${\vh}^{(k)}_v$ depends on $\vh^{(k-1)}_{v,\mathrm{exact}}$ and $\vh^{(k-1)}_{u,\mathrm{exact}} \; \forall u\in \cN_v$. This means the final feature representation vector for each node $v$, ${\vh}^{(K)}_v$, recursively depends on the feature representations of node $v$'s $K$-hop neighbors. For large, dense graphs, 
computing ${\vh}^{(K)}_v$ can be computationally expensive. One way to reduce these computational costs is by neighbor sampling \cite{NS_GraphSAGE}. Instead of using all neighbors $u\in  \cN_v$ of node $v$ to compute its representation at the next layer, at each layer $k$ we can select a small subset of node $v$'s neighbors $\widehat{\cN}^{(k)}_v\subset \cN_v$, where $|\widehat{\cN}^{(k)}_v|=D^{(k)}$. 

If $|\cN_v|$ is big, we can fix $D^{(k)}$ such that $D^{(k)}<<|\cN^{(k)}_v|$. 
To incorporate this neighbor sampling into the GCN framework $\vP$ in \eqref{eq:GCN1} is replaced with $\widehat{\vP}^{(k)}$, where
\begin{equation}\label{eq:hat-vP}
    \widehat{\vP}^{(k)}[v,u]=
    \begin{cases}
       \frac{|\cN_v|}{D^{(k)}}\vP[v,u]& \text{if } u\in \widehat{\cN}^{(k)}_v,\\[1mm]
       0 &\text{otherwise}.
    \end{cases}
\end{equation}
Let $\cV_B\subset \cV$ be a batch of nodes. The receptive field of $\cV_B$ at layer $k$ is defined as the set of nodes at layer $k$ whose feature representations are 
used to compute ${\vh}^{(K)}_v,\; \forall v\in \cV_B$. At the final layer $K$, $\vr^{(K)}_{\cV_B}$ is equal to $\cV_B$.  In the case of neighbor sampling, the receptive field $\vr^{(k)}_{\cV_B}$ at layer $k$ consists of the nodes in $\vr^{(k+1)}_{\cV_B}$ and their sampled neighbors at that layer. When nodes are sampled in a layer, as determined by \eqref{eq:hat-vP}, they are added to the layer's receptive field.

\subsection{Control Variate Estimator} 
Using features of sampled neighbors and skipping all other neighbor representations can cause a large deviation from the exact feature representation of a node. To address this issue, \cite{VRGCN} introduces the control variate estimator (CVE) that re-uses old feature representation of non-selected neighbors in the recursive computation.

Since the GCN feature representations are no longer exact when using this estimator, we use $\vH^{(k)}$ instead of $\vH^{(k)}_{\mathrm{exact}}$ to denote the node feature representations at layer $k$. Similar to  \eqref{eq:GCN1} and \eqref{eq:GCN2}, each feature representation $\vh^{(k)}_v$ is computed recursively from node $v$'s representations and those of its neighbors at each layer $k$. To mitigate large deviations from the exact feature representation, CVE maintains a matrix of historical feature representations $\overline{\vH}^{(k)}$ at each layer. The idea is that the feature representation vectors $\vh^{(k)}_v$ are only computed for nodes in the receptive field at that layer to save the computation time. The approximation $\bar{\vh}^{(k)}_v$  is updated every time $\vh^{(k)}_v$ is computed, i.e. for $v\in \vr^{(k)}_{\cV_B}$. The difference between the recursively computed $\vH^{(k)}$ and its historical approximation $\overline{\vH}^{(k)}$ is defined as $$\Delta \vH^{(k)}:=\vH^{(k)}-\overline{\vH}^{(k)}.$$

The more affordable version of the feature representation by using the CVE %proposed by Chen et al. 
is defined as follows for $k=0,\dots,K-1$:
\begin{align}
    \vZ^{(k+1)}&= \big(\widehat{\vP}^{(k)}\Delta \vH^{(k)}+\vP\overline{\vH}^{(k)} \big)\vW^{(k)}, \label{eq:CV1}\\ 
\vH^{(k+1)}&=\sigma(\vZ^{(k+1)}).\label{eq:CV2}
\end{align}
If the weight matrices $\{\vW^{(k)}\}$ do not change too quickly from iteration to iteration, it is expected that $\vH^{(k)}$ and $\overline{\vH}^{(k)}$ will be close to each other. When $\vH^{(k)}=\overline{\vH}^{(k)}$, the feature representation in \eqref{eq:CV1} and \eqref{eq:CV2} becomes the exact one in \eqref{eq:GCN1} and \eqref{eq:GCN2}, since $\Delta \vH^{(k)}=\textbf{0}$. In \eqref{eq:CV1}, $\Delta \vH^{(k)}$ is multiplied by $\widehat{\vP}^{(k)}$ which performs neighbor sampling to save computation time. However, neighbor sampling is not applied to $\overline{\vH}^{(k)}$. This will not cause a high computational cost because $\overline{\vH}^{(k)}$ is a historical approximation that is not computed recursively. 

\subsection{Proposed Adam-type Methods with CVE}\label{sec:alg}
The work \cite{VRGCN} provides convergence results for their control variate algorithm, which uses SGD to update the GCN's parameters based on the control variate gradient estimator. 
 However, it has been demonstrated extensively that SGD converges significantly slower than Adam-type methods on training deep learning models with Adam serving as a popular optimizer for training GCNs \cite{chen2018fastgcn,VRGCN}.

To achieve fast and guaranteed convergence, we propose Adam-type methods for training GCNs by utilizing the CVE. We present our methods in
Algorithm~\ref{alg:CV Training}.
Randomness is present in the algorithm in the minibatch sampling and in the neighbor sampling. The neighbor sampling is encoded in the  $\{\widehat{\vP}^{(k)}\}$ matrices, which are defined in \eqref{eq:hat-vP}.
We use the CVE technique given in  \eqref{eq:CV1} and \eqref{eq:CV2} to approximate node features at each layer $k=1,\ldots,K$. After $K$ layers, the final node representations $\{{\vh}^{(K)}_v\}_{v\in\cV_t}$ from the sampled minibatch are used to compute the minibatch loss $\ell$. 
At the $t$-th iteration, the trainable weight matrices are in the following format: $$\vW_t= \{\vW^{(0)}_t,\vW^{(1)}_t,\ldots,\vW^{(K-1)}_t\}.$$

The stochastic gradient over minibatch $\cV_B$ that is used to update $\vW_t$ is computed as follows:
\begin{align}\label{eq:stochgrad}
  \vG_t=  \frac{1}{|\cV_{B}|}\sum_{v\in \cV_{B}}\nabla_{\vW_t} f(\vy_v,{\vh}^{(K)}_v) 
\end{align}
These matrices are updated by the generalized Adam-type method, such as Heavy-Ball SGD, AMSGrad, and AdaGrad, and SGD, which have different settings of $\vV_t$ as shown in Table~\ref{tab:algs}. 

The historical approximation matrices $\{\overline{\vH}^{(k)}\}_{k=0}^{K-1}$ are updated according to the CVE method \cite{VRGCN}. The approximation for a node's representation at a particular layer is updated whenever that node is included in the receptive field at that layer. In other words, the historical approximation of a node's feature vector is set to its most recent recursively computed feature vector at that layer. This way, we can simultaneously achieve a low approximation error and reduce the computation cost by avoiding recursively computing all neighbors' \emph{exact} feature representation. The use of $\overline{\vH}^{(k)}$ is critical and enables us to establish convergence of our algorithm without requiring the unbiasedness of the stochastic gradients.% $\{\vG_t\}$ that we use. 

%%%%%%%%%%%%%%%%%%%%%%%%%%%%%%%%%%%%%%%

\begin{algorithm}[h]
\caption{CVE Adam-type methods for solving~\eqref{eq:loss}}
\label{alg:CV Training}
\begin{algorithmic}[1]
\State \textbf{Input:} \textnormal{node feature matrix $\vX$, normalized adjacency matrix $\vP$, total number of iterations $T$, learning rate $\alpha>0$, and momentum parameter $\beta_1\in [0,1)$ }
\State Initialize: $\vW_1,\vM_{0}=\textbf{0},\vH^{(0)}=\vX,\overline{\vH}^{(0)}=\vX, \overline{\vH}^{(k)}=\vP \overline{\vH}^{(k-1)}\vW_1 \;\forall\, k=1,\ldots,K-1$%, $t=1$ 
\For{$t=1,\ldots,T-1$}
    \State Take a minibatch $\cV_t \subset \cV_{\cL}$ (sampling with replacement) and perform
    \Statex \hspace{.5cm}neighbor sampling
        \State Compute receptive fields $\{\vr^{(k)}\}$ and stochastic propagation matrices $\{\widehat{\vP}^{(k)}\}$
        \Statex \hspace{.5cm}based on $\cV_t$ and sampled neighbors
        
        \For{$k=0,\ldots,K-1$ }
            \State $\vZ^{(k+1)}=\bigg(\widehat{\vP}^{(k)}(\vH^{(k)}-\overline{\vH}^{(k)})+\vP\overline{\vH}^{(k)} \bigg)\vW^{(k)}_t$
            \State $\vH^{(k+1)}=\sigma(\vZ^{(k+1)})$
        \EndFor
    \State Compute the minibatch loss 
    $$\ell(\vW_t)=\frac{1}{|\cV_{t}|}\sum_{v \in\cV_{t}}f(\vy_v,\vh_v^{(K)})$$
    \hspace{.5cm}where $\vh_v^{(K)}=\vH^{(K)}[v,:]^\top$.
    \State Compute the minibatch gradient $\vG_t=\nabla_{\vW_t}\ell$
    \State Update parameters by 
    \begin{align*}
    &\vM_t=\beta_1 \vM_{t-1}+(1-\beta_1)\vG_{t},\\ &\vV_{t}=h_t(\vG_{1},\ldots,\vG_{t}),\\
    &\vW_{t+1}=\vW_t-\alpha \vM_{t}/\sqrt{\vV_{t}}.
    \end{align*}
    \vspace{-4mm}
    \State Update historical approximations:
    \For{$k=0,\ldots,K-1$} 
        \For{$v\in \vr^{(k)}$}
        \State $\Bar{\vh}_v^{(k)}= \vh_v^{(k)}$
        \EndFor
    \EndFor
\EndFor
\State\textbf{Return} $\vW_{\tau}$, where $\tau$ is selected from $\{1,\ldots,T\}$ uniformly at random 
\end{algorithmic}
\end{algorithm}

\section{Related Works}\label{sec:relatedwork}

%\subsection{Adam Type Update}

%\subsection{GNN Models}
The GNN model was introduced by \cite{gori2005new,scarselli2008graph}. Extensions of these GNN models include Gated GNNs \cite{li2015gated}, GraphESN \cite{gallicchio2010graph}, and stochastic steady-state embedding (SSE) \cite{dai2018learning}. These methods are classified as Recurrent GNNs and use message passing to exchange information between graph nodes \cite{surveypaper}.

More recent work applies the concept of convolutions to graph data by aggregating neighbor information. Spectral-based methods perform these convolutions using the normalized graph Laplacian matrix \cite{surveypaper}. Examples of these types of methods include the Spectral CNN \cite{bruna2013spectral}, the deep CNN on graph data method \cite{henaff2015deep}, fast localized spectral filtering \cite{defferrard2016convolutional}, the CayleyNet method \cite{levie2018cayleynets}, and Adaptive GCN \cite{li2018adaptive}. 

The graph convolutional network model (GCN) \cite{OriginalGCN} performs graph convolutions using the normalized graph adjacency matrix, which aggregates feature representation vectors of nodes and their neighbors recursively throughout the layers of the network. The DualGCN \cite{zhuang2018dual} model incorporates an additional convolution to the traditional GCN, which is based on random walks. \cite{gao2018large} also extend upon traditional GCNs with their large-scale learnable GCN model.

Several variations of the GCN architecture have been introduced that use neighbor sampling to reduce the number of recursive computations required by traditional GCNs and improve their scalability. GraphSAGE \cite{NS_GraphSAGE} uniformly samples node neighbors to reduce the receptive field size. FASTGCN \cite{chen2018fastgcn} uses importance sampling to sample neighbors. LADIES \cite{LADIESzou2019layer} uses layer-dependent importance sampling. 
MG-GCN \cite{huang2020mg} introduces a degree-based sampling method, which performs neighbor sampling from different layers. \cite{huang2018adaptive} uses adaptive layer-wise sampling to address the scalability issue caused by large receptive field sizes across GCN layers. 

A few other methods use graph sampling as opposed to neighbor sampling to make training GCNs less computationally expensive. Cluster-GCN \cite{chiang2019cluster} and GraphSAINT \cite{zeng2019graphsaint}
both sample subgraphs to use for GCN training, as opposed to training on the entire graph. PromptGCN \cite{ji2024promptgcn} extends upon these graph sampling methods by allowing information sharing between subgraphs. SSGCN \cite{wang2024ssgcn} integrates graph convolutions across multiple minibatch learners. 

Among more recent methods in GCN training, Bi-GCN \cite{wang2021bi} reduces memory requirements by binarizing node feature vectors and GCN parameters. Label-GCN \cite{bellei2021label} modifies the traditional GCN architecture to propagate label information by eliminating self-loops. GCN-SL \cite{jiang2021gcn} and GCN-SA \cite{jiang2024self} modify the GCN architecture for better performance on graphs with low homophily. \cite{huang2024geometric} proposes GLGCN, integrating GCNs with multi-view learning for image data. 
\cite{cong2103importance} proposes and analyzes SGCN++,  a doubly variance reduction method for GCN training that can be applied to different GCN sampling algorithms, including the CVE method.

Other than developing GCN training methods, many other papers develop GCN models for specific applications. STFGCN \cite{ma2024spatiotraffic}, ST-DAGCN \cite{liu2024sttraffic}, and STIDGCN \cite{liu2024spatialtraffic} utilize GCNs for traffic forecasting. The IP-GCN model \cite{ali2024ip} predicts insulin protein for diabetes drug design. \cite{wang2023seerecommender} proposes KDGCN-IC and KDGCN-DC, which apply GCNs to recommender systems.

To the best of our knowledge, none of these works have established guaranteed convergence for Adam-type methods on training GCNs, though some papers, e.g., \cite{VRGCN}, demonstrate the empirical performance of Adam with neighbor sampling technique.

%%%%%%%%%%%%%%%%%%%%%%%%%%%%%%%%%%%%%%%%%%%%%%%%%
\section{Convergence Results}\label{sec:convergence}

In this section, we present our convergence results for Algorithm~\ref{alg:CV Training}. % that uses several different Adam-type optimizers. 
Due to the nonconvexity of the objective function $F$ in \eqref{eq:loss}, we do not expect to find a global optimal solution. Instead we show the convergence to a stationary solution by bounding $\EE[\|\nabla F(\vW_\tau)\|^2]$, where $\vW_\tau$ is the output of Algorithm~\ref{alg:CV Training} after $T$ iterations.

Assuming unbiased stochastic gradient, the work \cite{AdamAnalysis} provides convergence analysis for several Adam-type methods, 
which includes AMSGrad, Heavy-Ball SGD, SGD, and AdaGrad. The unbiasedness assumption does not hold for training GCNs using the neighbor sampling and/or CVE techniques; explained below. Hence, the results in \cite{AdamAnalysis} do not apply to a CVE-based Adam-type method. Though \cite{ajalloeian2020convergence} %provides a %more general complexity analysis of 
analyzes biased stochastic gradient methods, its assumption on the bias does not hold for the CVE-based Adam-type method. %the the way in which they bound the bias differs from ours.

%When the stochastic gradient used in the update is unbiased, \cite{AdamAnalysis} performs the convergence analysis for generalized Adam-type updates on solving a nonconvex stochastic optimization problem. However, their results do not apply to Algorithm~\ref{alg:CV Training} with Adam-updates, because the stochastic gradient $\vG_t$ we use is not unbiased. 
Our stochastic gradient is biased due to the neighbor sampling and the nonlinear activation function $\sigma(\cdot)$. Though $\widehat{\vP}^{(k)}$ defined in \eqref{eq:hat-vP} is an unbiased estimator of $\vP$, after $\vZ^{(k+1)}$ is passed through the non-linear activation function $\sigma$ in \eqref{eq:CV2}, $\vH^{(k+1)}$ is no longer an unbiased estimator of $\vH_{\mathrm{exact}}^{(k+1)}$, because $\EE[\sigma(\vZ)]\neq \sigma(\EE[\vZ])$ in general. Therefore $\vG_t$ is not an unbiased estimator of $\nabla F(\vW_t)$.

To address the challenge caused by the use of biased gradient, we apply a result in \cite{VRGCN} and show that 
\begin{align}
    \left\|\EE\big[\vG_t-\nabla F(\vW_t)\,|\vW_t\big]\right\| = O(\alpha)\label{eq:lemma6},
\end{align}
 which is stated in Lemma~\ref{lem:delta_bound} in Appendix~\ref{appendixA1}. While Lemma~\ref{lem:delta_bound} is established for the CVE-based stochastic gradient, other biased stochastic methods can also have such a bound such as the block gradient method in \cite{xu2015block}. 
 We highlight that our convergence results established in this section for Adam-type methods hold if the stochastic gradient satisfies \eqref{eq:lemma6}, in which case our results generalize to other applications and are not restricted to the GCN training problem. 
 This bound enables us to further bound, by $O(\alpha^2)$, a cross-product term that involves the stochastic error of $\vG_t$; see~\eqref{eq:36}. Notice that $O(\alpha^2)$ is often a dominating term about stochastic variance while analyzing stochastic gradient-type methods. Hence, we can still show convergence rates of Algorithm~\ref{alg:CV Training}, which are as good as or even better than the results in \cite{AdamAnalysis}, even though biased stochastic  gradients are used in our algorithm.    

\subsection{Assumptions}\label{sec:assumptions}
We make the following assumptions for our analysis. These assumptions are standard for analyzing Adam-type methods. 
Notice that we do not assume unbiasedness of $\vG_t$ for each iteration $t$. 
\begin{assumption}[smoothness]
\label{assum:smooth}
The objective function $F$ in \eqref{eq:loss} is differentiable and has a $\rho$-Lipschitz gradient, i.e. $$\|\nabla F(\vW)-\nabla F(\widehat{\vW}) \|\leq \rho \|\vW-\widehat{\vW} \| ,\;\forall \;\vW,\widehat{\vW}.$$
\end{assumption}
\begin{assumption}\label{assum:low-bd}
There exists a positive constant $\nu_{\mathrm{min}}$ such that
     $$(\vV_{t})_j \geq \nu_{\mathrm{min}}^2>0, \forall j.$$
\end{assumption}
Here, notice that $\vV_t$ has the same format as $\vW$, i.e., containing a set of matrices. The above condition should read as that each component of $\vV_t$ is lower bounded by $\nu_{\mathrm{min}}^2$.
\begin{assumption}[bounded gradients]
\label{assum:bound-grad}
The gradient of the function $F$ in \eqref{eq:loss} and the used stochastic gradients are bounded. More precisely, there exist positive constants $H_F$, $H_\infty$, and $H_1$ such that for each $t\ge0$,
\begin{align*}
&\|\nabla F(\vW_t) \|\leq H_F,\ \|\vG_{t} \|\leq H_F,\\
&\|\nabla F(\vW_t) \|_{\infty}\leq H_{\infty},\  \|\vG_t \|_{\infty}\leq H_{\infty},\\
&\|\nabla F(\vW_t) \|_{1}\leq H_{1},\ \|\vG_t \|_{1}\leq H_{1}.
\end{align*}
\end{assumption}

The condition $(\vV_{t})_j \geq \nu_{\mathrm{min}}^2>0, \forall j$ in Assumption~\ref{assum:low-bd} is required to avoid zero division. It automatically holds for certain settings such as for SGD and Heavy-Ball SGD method in Table~\ref{tab:algs}. It can also easily hold for other settings such as for AMSGrad if every entry of $\hat\vV_0$ is no less than $\nu_{\mathrm{min}}^2$. For AdaGrad in Table~\ref{tab:algs}, the obtained $\vV_t$ may not satisfy such a condition. Nevertheless, we can slightly change the $\vW$-update to $\vW_{t+1}=\vW_t-\alpha \vM_{t}/\sqrt{\vV_{t}+\nu_{\mathrm{min}}^2}$, and our analysis still holds with slight modifications.

\subsection{A General Case}
%%%%%%%%%%%%%%%%%%%%%%%%%%%%%%%%%%%%%%%%%%%%%%%%%%%%%%%

We first establish a general result without specifying the choice of the function $h_t$ in setting $\vV_t$. This result is the key to show convergence of several specific Adam-type methods listed in Table~\ref{tab:algs}. The proof of this Theorem is included in Appendix~\ref{appendixA1}. We modify the analysis of \cite{AdamAnalysis} by accommodating the biased stochastic gradient. %, while making the appropriate changes for our biased stochastic gradient case.

%%%%%%%%%%%%%%%%%%%%%%%%%%%%%%%%%%%%%%%%%%%%%%%%
\begin{theorem}[Key inequality]\label{thm:main_result}
Under Assumptions~\ref{assum:smooth}--\ref{assum:bound-grad}, let $\{\vW_t\}$ and $\{\vV_t\}$ be generated from Algorithm~\ref{alg:CV Training}, then it holds
\begin{equation}\label{eq:main-result}
\begin{aligned}
 \EE\bigg[\sum_{i=1}^T \ai\bigg\langle \nabla F(\vW_i) ,\frac{\nabla F(\vW_i)}{\svi} \bigg\rangle  \bigg]
&\leq C_1\EE \left[\sum_{i=2}^T\left\lVert\frac{\alpha}{\sqrt{\vV_i}}-\frac{\alpha}{\sqrt{\vV_{i-1}}} \right\rVert_1\right] \\
&+C_2\EE \left[\sum_{i=2}^T \left\lVert\frac{\alpha}{\sqrt{\vV_i}}-\frac{\alpha}{\sqrt{\vV_{i-1}}} \right\rVert^2\right]\\
&+C_3\alpha^2T+C_4\alpha+\EE[F({\vW}_1)-F^*],
\end{aligned}
\end{equation}
where $C_1,C_2,C_3,C_4$ are constants independent of $T$ and $\alpha$, and they are defined as follows: 
\begin{subequations}
    \label{eq:constants}
\begin{align}
C_1&=H_{\infty}^2\frac{\beta_1}{1-\beta_1}+2H_{\infty}^2
        %+2H_{\infty}^2 , 
        \label{eq:C1}\\
C_2&=\rho\bigg(\frac{\beta_1}{1-\beta_1} \bigg)^2H_{\infty}^2, \label{eq:C2}\\
    C_3&=\frac{\rho H_F^2}{\nu_{\mathrm{min}}^2} +\frac{H_F^2}{2\nu_{\mathrm{min}}^2}\bigg(\rho^2\frac{\beta_1^2}{(1-\beta_1)^2}+1 \bigg)+\frac{CH_1H_{\infty}}{\nu_{\mathrm{min}}^2}, \label{eq:C3}\\
    C_4&=\frac{H_1H_{\infty}}{\nu_{\mathrm{min}}}, \label{eq:C4}
        \\
\end{align}
\end{subequations}
with $C$ being a universal constant.
\end{theorem}

Recall that $\vW_\tau$ is the output of Algorithm~\ref{alg:CV Training} after $T$ iterations. Thus to establish convergence to stationarity (in expectation), we need $\EE[\|\nabla F(\vW_\tau)\|^2] \to 0$ as $T\to \infty$. Since $\tau$ is selected from $\{1,\ldots,T\}$ uniformly at random, it holds
\begin{align*}
  \,  T \alpha \EE\left[\bigg\langle \nabla F(\vW_\tau) ,\frac{\nabla F(\vW_\tau)}{\sqrt{\vV_\tau}} \bigg\rangle\right]
  =\,\EE\bigg[\sum_{i=1}^T \ai\bigg\langle \nabla F(\vW_i) ,\frac{\nabla F(\vW_i)}{\svi} \bigg\rangle  \bigg].
\end{align*} 
Hence, if $\vV_\tau$ is upper bounded entrywise  
and the right hand side of \eqref{eq:main-result} is upper bounded by a constant, we can show that $\EE[\|\nabla F(\vW_\tau)\|^2] = O(\frac{1}{T\alpha})$. The convergence rate results  in the next subsection will be established by adopting this idea.

%%%%%%%%%%%%%%%%%%%%%%%%%%%%%%%%%%%%%%%%%%%%%%%%%
\subsection{Several Specific Optimizers}
In this subsection, we give a few specific choices of $\vV_t$ and show the convergence results of Algorithm~\ref{alg:CV Training} under these choices.
%we present convergence results for the variations of Algorithm~\ref{alg:CV Training} defined 
In Table \ref{tab:algs}, we list four optimizers and the corresponding choices of $\beta_1$ and $\vV_t$, as well as the convergence rate results measured on $\EE[\|\nabla F(\vW_\tau)\|^2]$.

\begin{table}[h]
\caption{Settings of $\beta_1$ and $\{\vV_t\}$ for different optimizers and their corresponding convergence rates.} \label{tab:algs}%
\begin{tabular}{@{}llll@{}}
\toprule
        Optimizer & $\beta_1$ & $\vV_t$ & Rate\\ 
        \midrule
        SGD        & 0 & \textbf{1} & $\frac{1}{\sqrt{T}}$   \\
        Heavy-Ball SGD        &\textgreater 0 & \textbf{1}  & $\frac{1}{\sqrt{T}}$   \\
        \multirow{2}{*}{AMSGrad}    &   \multirow{2}{*}{\textgreater 0} & $\hat{\vV}_t=\beta_2\hat{\vV}_{t-1} +(1-\beta_2)\vG_t^2$ & \multirow{2}{*}{$\frac{1}{\sqrt{T}}$} \\
   & & $\vV_t=\max(\vV_{t-1},\hat{\vV}_t)$  &   \\
        AdaGrad &  \textgreater 0 & $\frac{1}{t}\sum_{i=1}^t\vG_i^2$ & $\frac{\log(T)}{T}+\frac{1}{\sqrt{T}}$\\
\botrule
\end{tabular}
\end{table}

%%%%%%%%%%%%%%%%%%%%%%%%%%%%%%%%%%%%%%%%%%%%%%%%%
\subsubsection{AMSGrad Convergence}

We first present the convergence rate result of Algorithm~\ref{alg:CV Training} by using the AMSGrad optimizer.

\begin{theorem}[AMSGrad Optimizer]\label{thm:AMSGrad}
Under Assumptions~\ref{assum:smooth}--\ref{assum:bound-grad}, let $\{\vW_t\}$ be generated from  
Algorithm~\ref{alg:CV Training} with $\alpha=\frac{\eta}{\sqrt{T}}$ for some $\eta>0$, $\beta_1\in (0,1)$ and $\vV_t$ by
$$\hat{\vV}_t=\beta_2\hat{\vV}_{t-1} +(1-\beta_2)\vG_t^2, \vV_t=\max(\vV_{t-1},\hat{\vV}_t)$$
for some $\beta_2\in (0,1)$. Then
\begin{align*}
   \EE \bigg[\|\nabla F(\vW_{\tau}) \|^2 \bigg]&\leq
   \frac{1}{T}\bigg(\frac{C_1dH_{\infty}}{\nu_{\mathrm{min}}}+C_4H_{\infty} 
  \bigg)\\
  &+\frac{1}{\sqrt{T}}\bigg(\frac{C_2\eta dH_{\infty}}{\nu_{\mathrm{min}}^2}+C_3\eta H_{\infty}+\frac{H_{\infty}}{\eta}\EE[F({\vW}_1)-F^*]\bigg), %\\
\end{align*}
where $C_1,\ldots,C_4$ are defined in \eqref{eq:constants}.
\end{theorem}

%%%%%%%%%%%%%%%%%%%%%%%%%%%%%%%%%
\begin{proof}
For the case of AMSGrad, $\vV_t$ is defined as follows:
\begin{align}\label{eq:update-V-AMSGrad}
    \hat{\vV}_t&=\beta_2\hat{\vV}_{t-1} +(1-\beta_2)\vG_t^2, 
    \vV_t=\max(\vV_{t-1},\hat{\vV}_t).
\end{align}
This means that $(\vV_t)_j\geq(\vV_{t-1})_j,\,\forall\, j$. 
Below we upper bound the first two terms in the RHS of \eqref{eq:main-result}.

To bound the first term in the RHS of \eqref{eq:main-result}, we have from $(\vV_t)_j\geq(\vV_{t-1})_j,\,\forall\, j$ that
\begin{align*}
    \EE\left[\sum_{t=2}^T \left\lVert\frac{\alpha}{\sqrt{\vV_t}}-\frac{\alpha}{\sqrt{\vV_{t-1}}}  \right\rVert_1\right]&=\EE\bigg[\sum_{j=1}^d\sum_{t=2}^T \bigg(\frac{\alpha}{(\sqrt{\vV_{t-1}})_j}-\frac{\alpha}{(\sqrt{\vV_{t}})_j}  \bigg)\bigg]\\
    &=\EE\bigg[\sum_{j=1}^d \bigg(\frac{\alpha}{(\sqrt{\vV_{1}})_j}-\frac{\alpha}{(\sqrt{\vV_{T}})_j}  \bigg)\bigg]\\
    &\leq \EE\bigg[\sum_{j=1}^d \frac{\alpha}{(\sqrt{\vV_{1}})_j} \bigg]\\
    &\leq \frac{d\alpha}{\nu_{\mathrm{min}}}.
\end{align*}

For the second term in the RHS of \eqref{eq:main-result}, it holds that
\begin{align}\label{eq:bd-2nd-term-main}
\EE\bigg[\sum_{t=2}^{T} \left\lVert\frac{\alpha}{\sqrt{\vV_t}}-\frac{\alpha}{\sqrt{\vV_{t-1}}}\right\rVert^2\bigg]&=\alpha^2\EE\bigg[\sum_{t=2}^{T}\sum_{j=1}^d \bigg( \frac{1}{(\sqrt{\vV_t})_j}-\frac{1}{(\sqrt{\vV_{t-1}})_j}  \bigg)^2   \bigg]\notag\\
&\leq \alpha^2\EE\bigg[\sum_{t=2}^{T}\sum_{j=1}^d \max\bigg( \frac{1}{(\sqrt{\vV_t})_j},\frac{1}{(\sqrt{\vV_{t-1}})_j}  \bigg)^2  \bigg]\notag\\
&\leq \alpha^2\EE\bigg[\sum_{t=2}^{T}\sum_{j=1}^d \frac{1}{\nu_{\mathrm{min}}^2}  \bigg]\notag\\
&=\frac{\alpha^2 d (T-1)}{\nu_{\mathrm{min}}^2}\notag\\
&\leq \frac{\alpha^2 d T}{\nu_{\mathrm{min}}^2}.
\end{align}

Using the above two bounds, the right side of \eqref{eq:main-result} can be bounded by
\begin{equation}\label{eq:up-bd-right-main}
\begin{aligned}
   &\EE \bigg[C_1\sum_{i=2}^T\left\lVert\frac{\alpha}{\sqrt{\vV_i}}-\frac{\alpha}{\sqrt{\vV_{i-1}}} \right\rVert_1 +C_2\sum_{i=2}^T \left\lVert\frac{\alpha}{\sqrt{\vV_i}}-\frac{\alpha}{\sqrt{\vV_{i-1}}} \right\rVert^2\bigg]\\
&+C_3\alpha^2T+C_4\alpha+\EE[F({\vW}_1)-F^*]\\
    &\leq C_1\frac{d\alpha}{\nu_{\mathrm{min}}}+C_2 \frac{\alpha^2 d T}{\nu_{\mathrm{min}}^2}+C_3\alpha^2T+C_4\alpha+\EE[F({\vW}_1)-F^*].
\end{aligned}
\end{equation}

Moreover, we have $\frac{1}{(\sqrt{\vV_t})_j}\geq \frac{1}{H_\infty}, \forall\, j$, which can be proved by induction using the assumption that $\|\vG_t\|_{\infty}\leq H_{\infty}$ and the update rule for $\vV_t$ in \eqref{eq:update-V-AMSGrad}. 
Hence, the left side of \eqref{eq:main-result} is lower bounded by
\begin{align}\label{eq:lw-bd-left-main}
   \EE\bigg[\sum_{t=1}^T \alpha \langle \nabla F(\vW_t),\nabla F(\vW_t)/\sqrt{\vV_t} \rangle\bigg]&\geq \frac{\alpha}{H_{\infty}} \EE \bigg[\sum_{t=1}^T\|\nabla F(\vW_t) \|^2 \bigg]= \frac{\alpha}{H_{\infty}} T\EE \bigg[\|\nabla F(\vW_{\tau}) \|^2 \bigg],
\end{align}
where the equality follows from $\tau \sim \{1,\ldots,T\}$ uniformly at random. 
Combining \eqref{eq:up-bd-right-main} and \eqref{eq:lw-bd-left-main} gives 
\begin{align*}
    \frac{\alpha}{H_{\infty}} T\EE \bigg[\|\nabla F(\vW_{\tau}) \|^2 \bigg]\leq C_1\frac{d\alpha}{\nu_{\mathrm{min}}}+C_2 \frac{\alpha^2 d T}{\nu_{\mathrm{min}}^2}+C_3\alpha^2T+C_4\alpha+\EE[F({\vW}_1)-F^*].
\end{align*}
Now dividing both sides of the above inequality by $\frac{\alpha}{H_{\infty}} T$ and using $\alpha=\frac{\eta}{\sqrt{T}}$ yields the desired result.    
\end{proof}

%%%%%%%%%%%%%%%%%%%%%%%%%
\subsubsection{Heavy-Ball SGD Convergence}

The next theorem is about the convergence rate of Algorithm~\ref{alg:CV Training} by using the Heavy-Ball SGD optimizer. As a special case, it also applies to the classic SGD that uses $\beta_1=0$.

\begin{theorem}[Heavy-Ball SGD Optimizer]\label{thm:SGD_HeavyBall}
Under Assumptions~\ref{assum:smooth}--\ref{assum:bound-grad}, let $\{\vW_t\}$ be generated from  
Algorithm~\ref{alg:CV Training} with $\alpha=\frac{\eta}{\sqrt{T}}$ for some $\eta>0$, $\beta_1\in [0,1)$ and $\vV_t=\vone, \forall\, t$. Then
\begin{align*}
    \EE \bigg[&\|\nabla F(\vW_{\tau}) \|^2 \bigg]
    \leq 
    \frac{C_4}{T}+\frac{1}{\sqrt{T}}\bigg(C_3\eta+\frac{1}{\eta}\EE[F({\vW}_1)-F^*]\bigg),
\end{align*}
where $C_3$ and $C_4$ are defined in \eqref{eq:C3} and \eqref{eq:C4}. %\tau \sim \{1,\ldots,T\}$.
\end{theorem}

\begin{remark}\label{rm:rate-amsgrad-hb}
When $T$ is sufficiently large, the terms involving $\frac{1}{\sqrt{T}}$ will dominate those of $\frac{1}{T}$ in Theorems~\ref{thm:AMSGrad} and~\ref{thm:SGD_HeavyBall}. Hence, we can simply write the convergence rate results as  $\EE \left[\|\nabla F(\vW_{\tau}) \|^2 \right]  = O(\frac{1}{\sqrt{T}})$. This convergence rate matches the lower bound result given in \cite{arjevani2023lower} for stochastic nonconvex optimization and thus is optimal. 
\end{remark}

\begin{proof}
    The Heavy-Ball SGD has $\vV_t=\vone, \forall t$. In this case,
the right side of \eqref{eq:main-result} becomes
\begin{align*}
  & \EE \bigg[C_1\sum_{i=2}^T\left\lVert\frac{\alpha}{\sqrt{\vV_i}}-\frac{\alpha}{\sqrt{\vV_{i-1}}} \right\rVert_1 +C_2\sum_{i=2}^T \left\lVert\frac{\alpha}{\sqrt{\vV_i}}-\frac{\alpha}{\sqrt{\vV_{i-1}}} \right\rVert^2\bigg]
+C_3\alpha^2T+C_4\alpha+\EE[F({\vW}_1)-F^*]\\
   = &\, C_3\alpha^2T+C_4\alpha+\EE[F({\vW}_1)-F^*]
\end{align*}
The left side of \eqref{eq:main-result} becomes
\begin{align*}
   \EE\bigg[\sum_{t=1}^T \alpha \langle \nabla F(\vW_t),\nabla F(\vW_t)/\sqrt{\vV_t} \rangle\bigg]&= \alpha \EE \bigg[\sum_{t=1}^T\|\nabla F(\vW_t) \|^2 \bigg]
   =\alpha T\EE \bigg[\|\nabla F(\vW_{\tau}) \|^2 \bigg].
\end{align*}
Combining the above two equations and plugging in $\alpha=\frac{\eta}{\sqrt{T}}$ gives the desired result. 
\end{proof}

%%%%%%%%%%%%%%%%%%%%%%%%%%%%%%%%%%%%%%%%%%%%
\subsubsection{AdaGrad Convergence}

The theorem below gives the convergence rate of Algorithm~\ref{alg:CV Training} with the AdaGrad optimizer.
\begin{theorem}[AdaGrad Optimizer]\label{thm:AdaGrad}
Under Assumptions~\ref{assum:smooth}--\ref{assum:bound-grad}, let $\{\vW_t\}$ be generated from  
Algorithm~\ref{alg:CV Training} with $\alpha=\frac{\eta}{\sqrt{T}}$ for some $\eta>0$, $\beta_1\in [0,1)$ and $\vV_t=\frac{1}{t}\sum_{i=1}^t \vG_i^2, \forall\, t$. Then
\begin{align*}
   \EE \bigg[\|\nabla F(\vW_{\tau}) \|^2 \bigg]&\leq \frac{\log(T)}{T}\bigg(\frac{C_1dH_{\infty}^3}{\nu_{\mathrm{min}}^3} \bigg)+
   \frac{1}{T}(C_4H_{\infty})\\
   &+\frac{1}{\sqrt{T}}\bigg(\frac{C_2\eta dH_{\infty}}{\nu_{\mathrm{min}}^2}+C_3\eta H_{\infty}+\frac{H_{\infty}}{\eta}\EE[F({\vW}_1)-F^*]\bigg), 
\end{align*}
where $C_1,\ldots,C_4$ are defined in \eqref{eq:constants}.
\end{theorem}
\begin{remark}
In the convergence rate result of Theorem~\ref{thm:AdaGrad}, there are terms involving $\frac{\log T}{T}$, $\frac{1}{T}$, and $\frac{1}{\sqrt{T}}$. When $T$ is sufficiently big, $\frac{1}{\sqrt{T}}$ will dominate both of terms involving $\frac{\log T}{T}$ and $\frac{1}{T}$. Hence, we can also state the result as $\EE \left[\|\nabla F(\vW_{\tau}) \|^2 \right]  = O(\frac{1}{\sqrt{T}})$. This convergence result is again optimal. Notice that in \cite{AdamAnalysis}, the convergence rate of AdaGrad (implied by that of AdaFom) is $O(\frac{\log T}{\sqrt{T}})$ because they adopt a diminishing stepsize instead of a constant stepsize.
\end{remark}

\begin{proof}
    For the AdaGrad optmizer, $\vV_t$ is defined as follows:
\begin{align*}
    \vV_t=\frac{1}{t}\sum_{i=1}^t\vG_i^2.
\end{align*}

To bound the first term in the RHS of \eqref{eq:main-result}, we have
\begin{align*}
   &\, \EE\bigg[\sum_{t=2}^T \left\lVert\frac{\alpha}{\sqrt{\vV_t}}-\frac{\alpha}{\sqrt{\vV_{t-1}}}  \right\rVert_1\bigg]=\alpha\EE\bigg[\sum_{j=1}^d\sum_{t=2}^T \bigg|\frac{1}{(\sqrt{\vV_{t}})_j}-\frac{1}{(\sqrt{\vV_{t-1}})_j}  \bigg|\bigg]\\
    = &\,\alpha\EE\bigg[\sum_{j=1}^d\sum_{t=2}^T \bigg|\frac{(\vV_{t-1})_j-(\vV_{t})_j}{(\sqrt{\vV_{t}})_j(\sqrt{\vV_{t-1}})_j\big((\sqrt{\vV_{t}})_j+(\sqrt{\vV_{t-1}})_j\big)}  \bigg|\bigg]\\
    \leq&\, \frac{\alpha}{2\nu_{\mathrm{min}}^3}\EE\bigg[\sum_{j=1}^d\sum_{t=2}^T \big|(\vV_{t-1})_j-(\vV_{t})_j  \big|\bigg]\\
    = &\,\frac{\alpha}{2\nu_{\mathrm{min}}^3}\EE\bigg[\sum_{j=1}^d\sum_{t=2}^T \bigg|\frac{1}{t-1}\sum_{i=1}^{t-1}(\vG_i)_j^2 -\frac{1}{t}\sum_{i=1}^{t}(\vG_i)_j^2 \bigg|\bigg]\\
    = &\,\frac{\alpha}{2\nu_{\mathrm{min}}^3}\EE\bigg[\sum_{j=1}^d\sum_{t=2}^T \bigg|\frac{1}{t(t-1)}\sum_{i=1}^{t-1}(\vG_i)_j^2 -\frac{1}{t}(\vG_t)_j^2 \bigg|\bigg]\\
    \leq&\, \frac{\alpha}{2\nu_{\mathrm{min}}^3}\EE\bigg[\sum_{j=1}^d\sum_{t=2}^T \bigg|\frac{1}{t(t-1)}\sum_{i=1}^{t-1}(\vG_i)_j^2\bigg| +\bigg|\frac{1}{t}(\vG_t)_j^2 \bigg|\bigg]\\
    \leq &\,\frac{\alpha}{2\nu_{\mathrm{min}}^3}\EE\bigg[\sum_{j=1}^d\sum_{t=2}^T \bigg|\frac{1}{t(t-1)}\sum_{i=1}^{t-1}H_{\infty}^2\bigg| +\bigg|\frac{1}{t}H_{\infty}^2 \bigg|\bigg]\\
    %&=\frac{H_{\infty}^2\alpha d}{2\nu_{\mathrm{min}}^3}\EE\bigg[\sum_{t=2}^T \bigg|\frac{1}{t(t-1)}(t-1)\bigg| +\bigg|\frac{1}{t}\bigg|\bigg]\\
    =&\,\frac{H_{\infty}^2\alpha d}{\nu_{\mathrm{min}}^3}\bigg[\sum_{t=2}^T \frac{1}{t}\bigg]\leq\frac{H_{\infty}^2\alpha d}{\nu_{\mathrm{min}}^3}\bigg[\int_1^T \frac{1}{x}\, dx\bigg]=\frac{H_{\infty}^2\alpha d}{\nu_{\mathrm{min}}^3}\log(T).
\end{align*}

For the second term in the RHS of \eqref{eq:main-result}, we use the bound obtained in \eqref{eq:bd-2nd-term-main}.

Hence, 
the right side of \eqref{eq:main-result} can be bounded by

\begin{align*}
   &\EE \left[C_1\sum_{i=2}^T\left\lVert\frac{\alpha}{\sqrt{\vV_i}}-\frac{\alpha}{\sqrt{\vV_{i-1}}} \right\rVert_1 +C_2\sum_{i=2}^T \left\lVert\frac{\alpha}{\sqrt{\vV_i}}-\frac{\alpha}{\sqrt{\vV_{i-1}}} \right\rVert^2\right]\\
&+C_3\alpha^2T+C_4\alpha+\EE[F({\vW}_1)-F^*]\\
    &\leq C_1\frac{H_{\infty}^2\alpha d}{\nu_{\mathrm{min}}^3}\log(T)+C_2 \frac{\alpha^2 d T}{\nu_{\mathrm{min}}^2}+C_3\alpha^2T+C_4\alpha+\EE[F({\vW}_1)-F^*].
\end{align*}

The left side of \eqref{eq:main-result} is lower bounded in \eqref{eq:lw-bd-left-main}. 
Combining this with the inequality above gives %the new bound on the right side of \eqref{thm:main_result}, we have:
\begin{align*}
    \frac{\alpha}{H_{\infty}} T\EE \bigg[\|\nabla F(\vW_{\tau}) \|^2 \bigg]\leq C_1\frac{H_{\infty}^2\alpha d}{\nu_{\mathrm{min}}^3}\log(T)+C_2 \frac{\alpha^2 d T}{\nu_{\mathrm{min}}^2}+C_3\alpha^2T+C_4\alpha+\EE[F({\vW}_1)-F^*].
\end{align*}

 Now plugging $\alpha=\frac{\eta}{\sqrt{T}}$ and dividing both sides by $\frac{\alpha}{H_{\infty}} T$ yields the desired result.

\end{proof}

%%%%%%%%%%%%%%%%%%%%%%%%%%%%%%%%%%%%%%%%%%%%%%%

\section{Numerical Experiments}\label{sec:numerical}
\begin{comment}
\begin{figure}[H]
\setlength\tabcolsep{1pt}
\begin{tabular}{ccc}
\includegraphics[width=0.3\textwidth]{results/Cora/train_acc.png} & \includegraphics[width=0.3\textwidth]{results/Cora/val_acc.png} & \includegraphics[width=0.3\textwidth]{results/Cora/test_acc.png}\\ \includegraphics[width=0.3\textwidth]{results/CiteSeer/train_acc.png} & \includegraphics[width=0.3\textwidth]{results/CiteSeer/val_acc.png} &
\includegraphics[width=0.3\textwidth]{results/CiteSeer/test_acc.png}\\ \includegraphics[width=0.3\textwidth]{results/ogbn-arxiv/train_acc.png} & \includegraphics[width=0.3\textwidth]{results/ogbn-arxiv/val_acc.png} & \includegraphics[width=0.3\textwidth]{results/ogbn-arxiv/test_acc.png}\\ \includegraphics[width=0.3\textwidth]{results/flickr/train_acc.png} &
\includegraphics[width=0.3\textwidth]{results/flickr/val_acc.png} & \includegraphics[width=0.3\textwidth]{results/flickr/test_acc.png} \\ \includegraphics[width=0.3\textwidth]{results/Reddit/train_acc.png} & \includegraphics[width=0.3\textwidth]{results/Reddit/val_acc.png} & \includegraphics[width=0.3\textwidth]{results/Reddit/test_acc.png} \\
\multicolumn{3}{c}{\includegraphics[width=0.7\textwidth]{results/legend.png}}
\end{tabular}
\caption{Training (first column), validation (second column), and testing (third column) curves for the datasets in Table~\ref{tab:datasets}. Each row corresponds to a dataset, the top to bottom order is Cora, CiteSeer, ogbn-arxiv, Flickr, Reddit. Here, ``HeavyBall'' refers to Heavy-Ball SGD}
\label{fig:accuracies}
\end{figure}

\end{comment}

\begin{figure}[H]
\setlength\tabcolsep{1pt}
\begin{tabular}{ccc}
\includegraphics[width=0.3\textwidth]{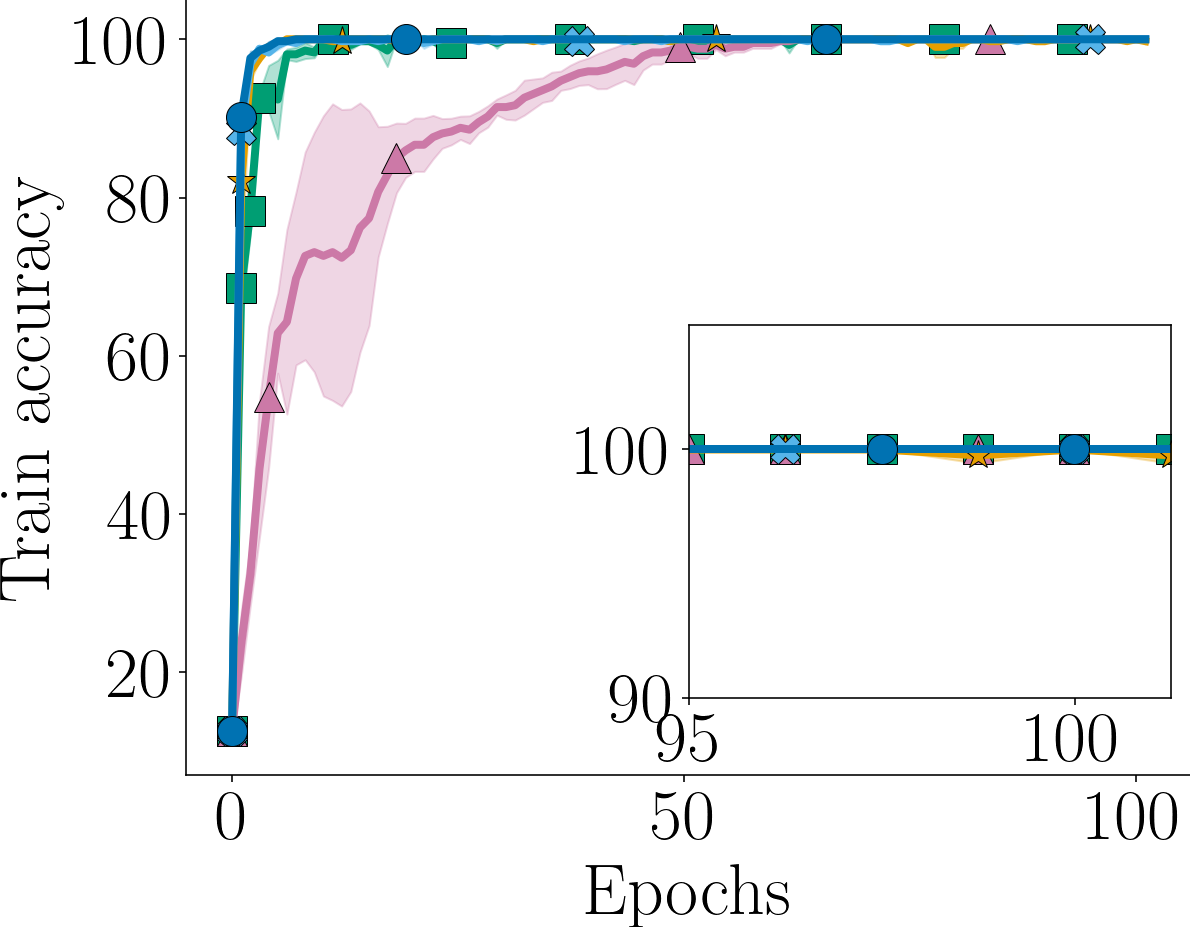} & \includegraphics[width=0.3\textwidth]{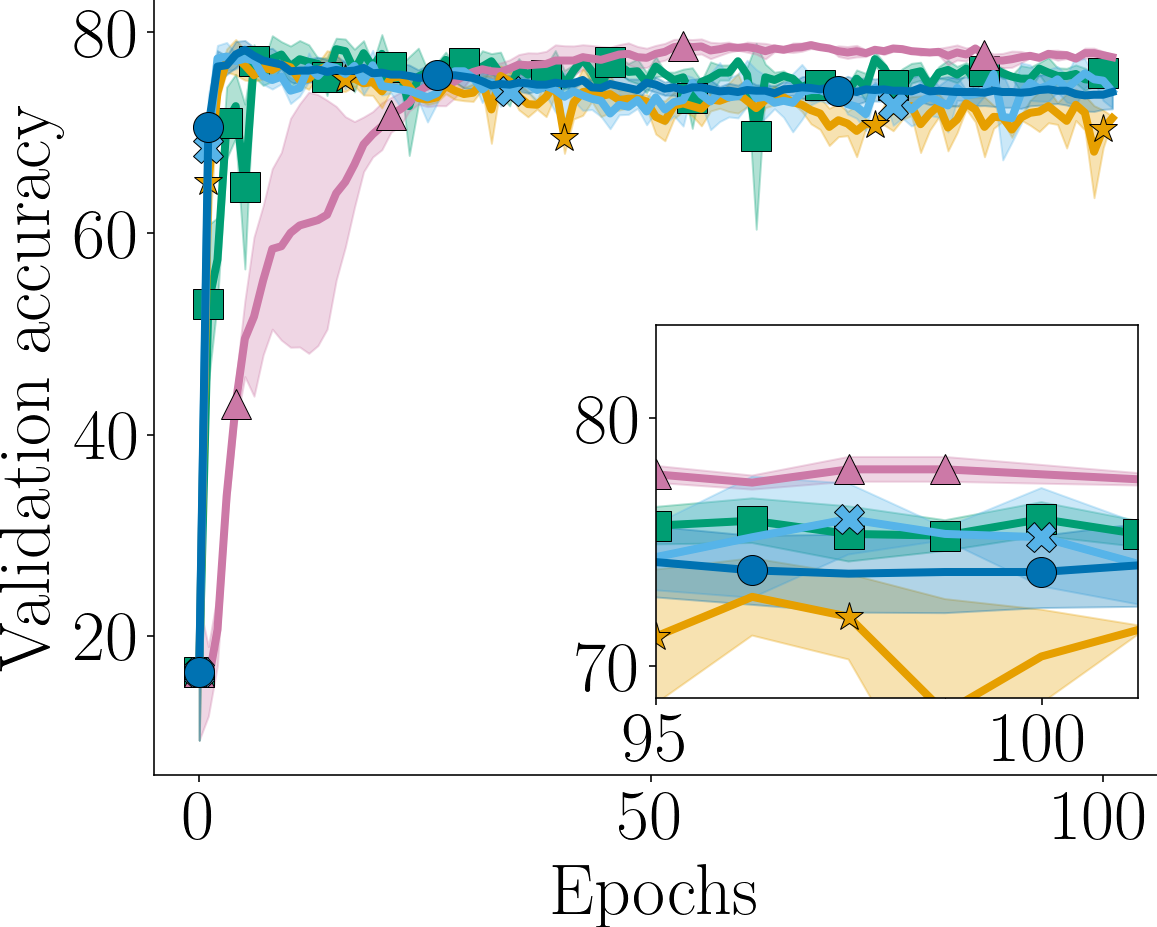} & \includegraphics[width=0.3\textwidth]{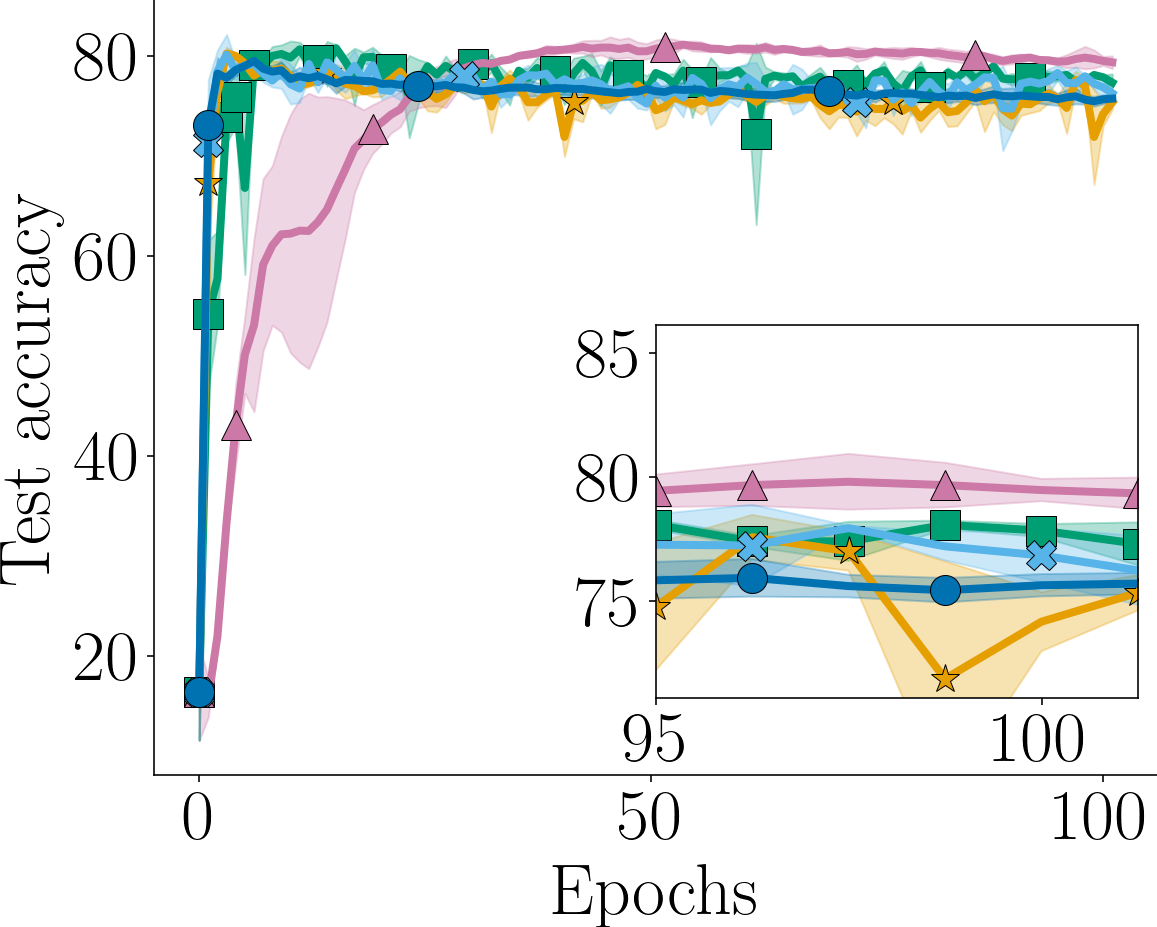}\\ \includegraphics[width=0.3\textwidth]{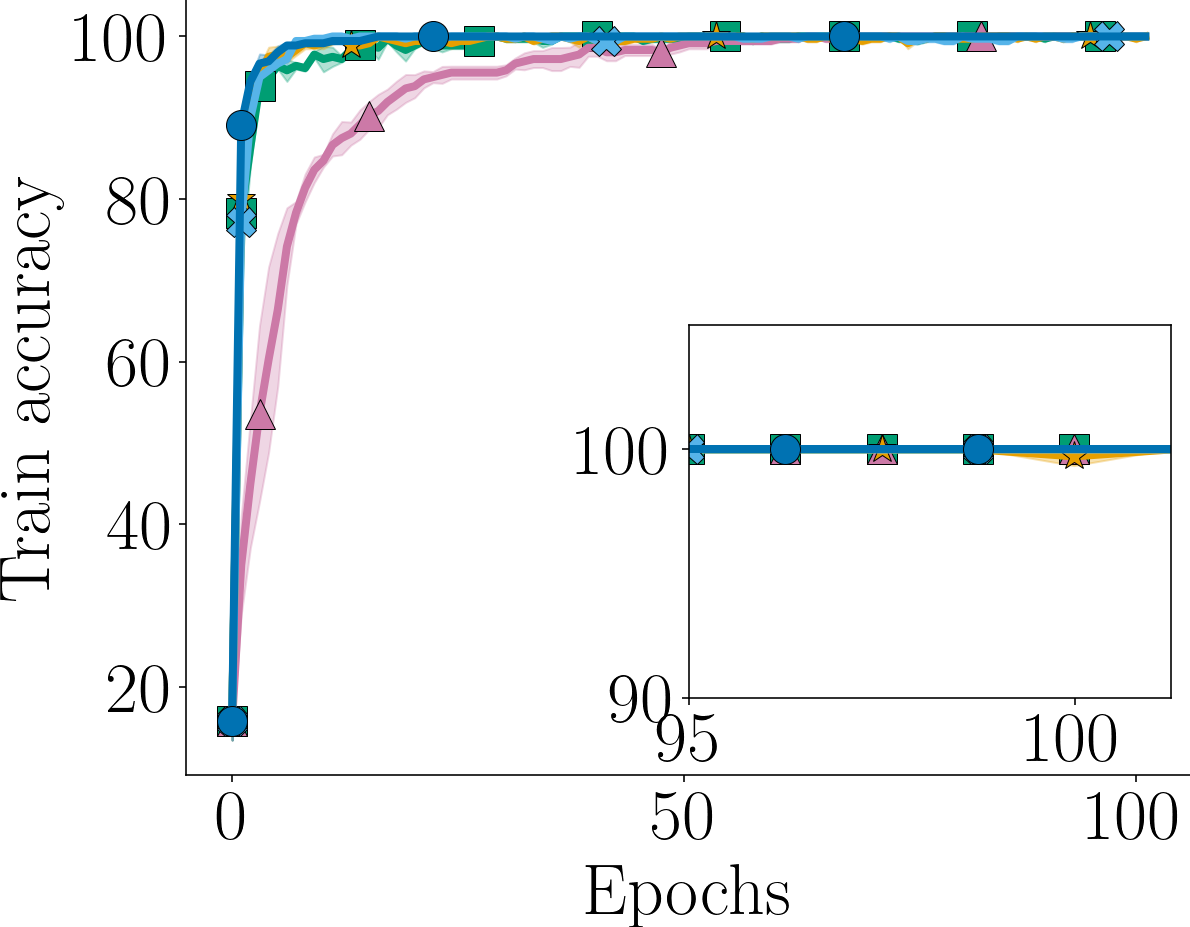} & \includegraphics[width=0.3\textwidth]{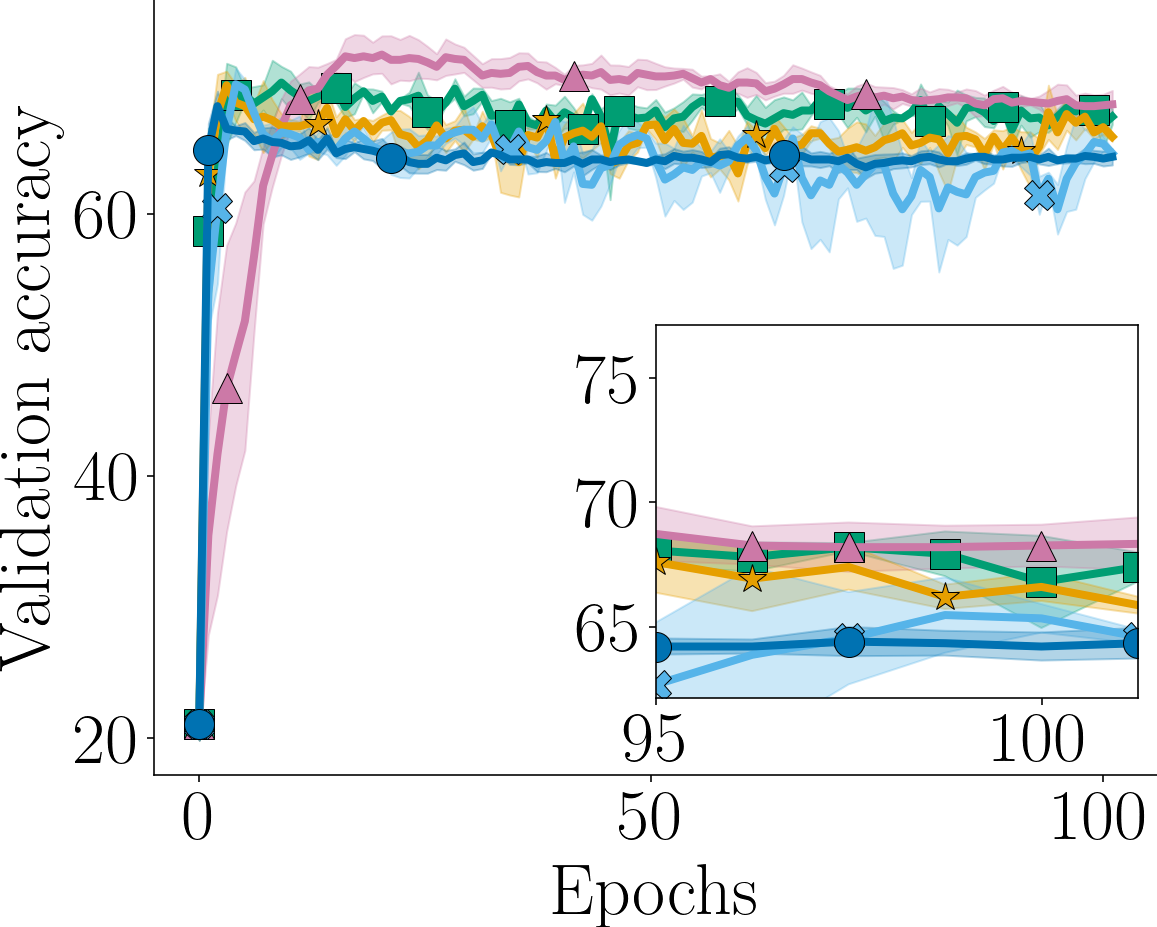} &
\includegraphics[width=0.3\textwidth]{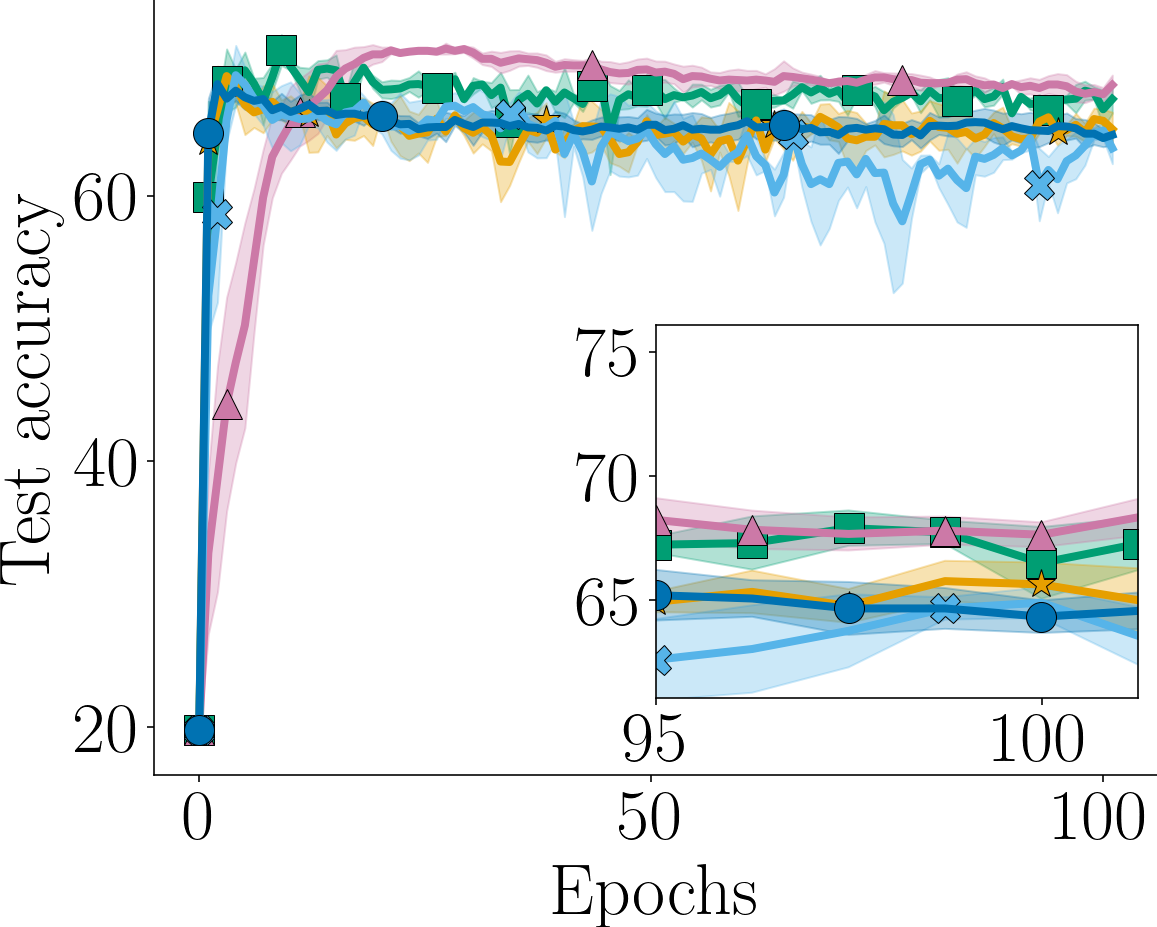}\\ \includegraphics[width=0.3\textwidth]{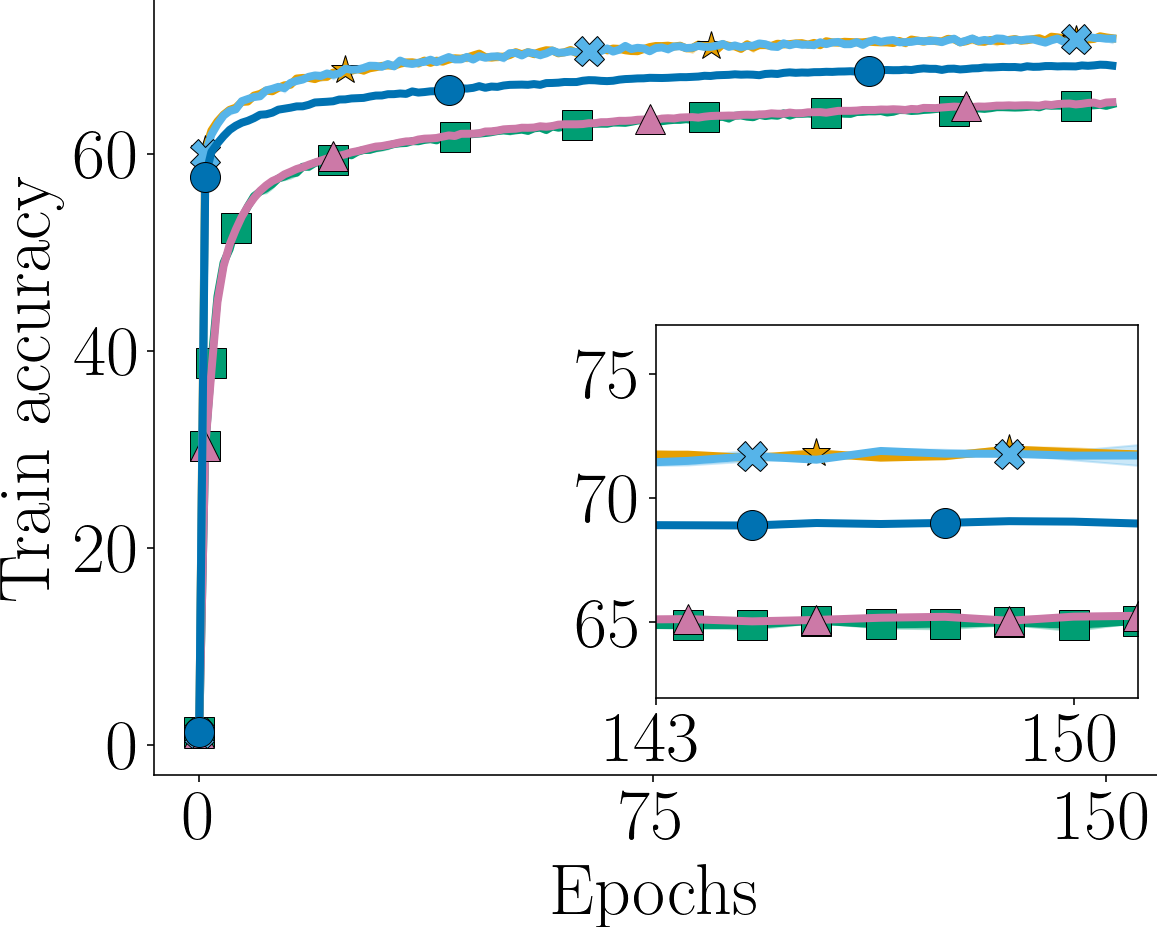} & \includegraphics[width=0.3\textwidth]{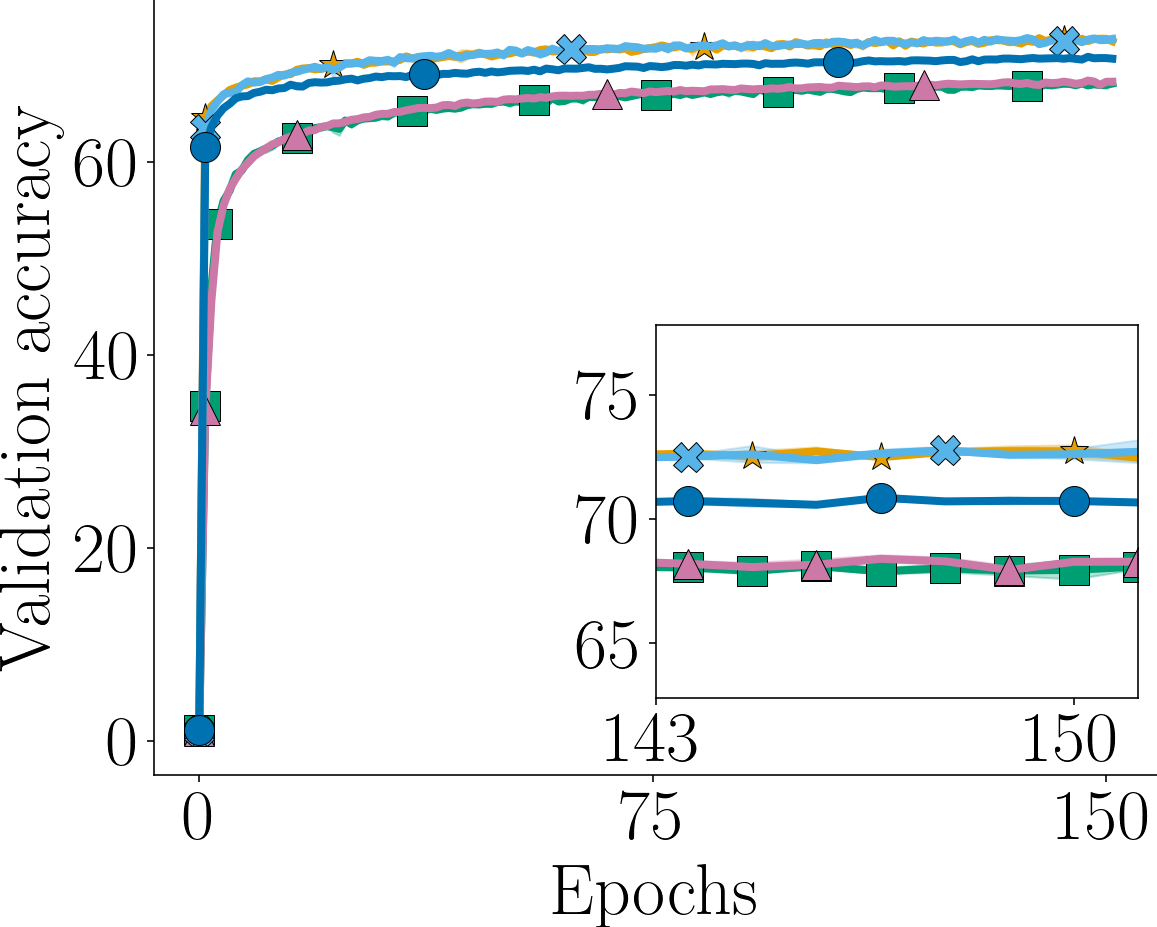} & \includegraphics[width=0.3\textwidth]{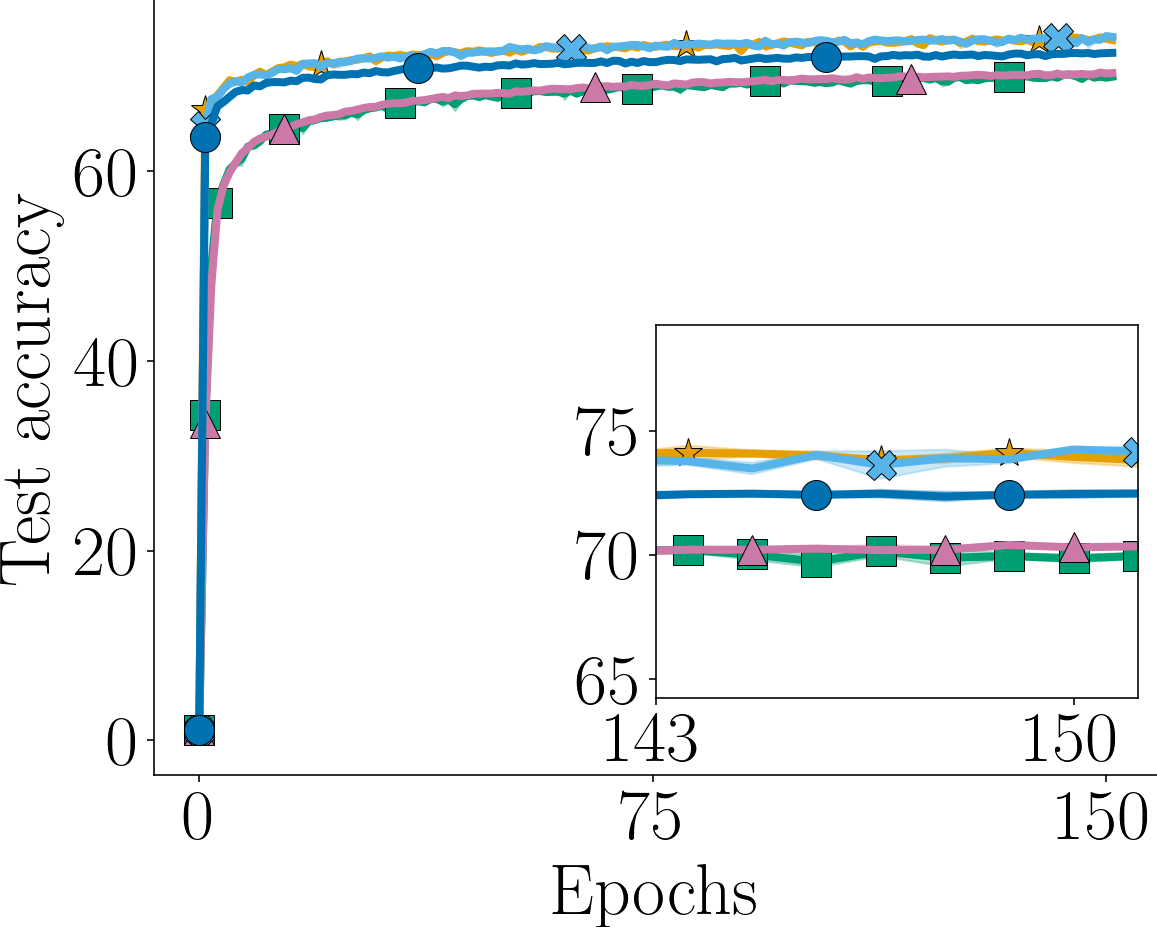}\\ \includegraphics[width=0.3\textwidth]{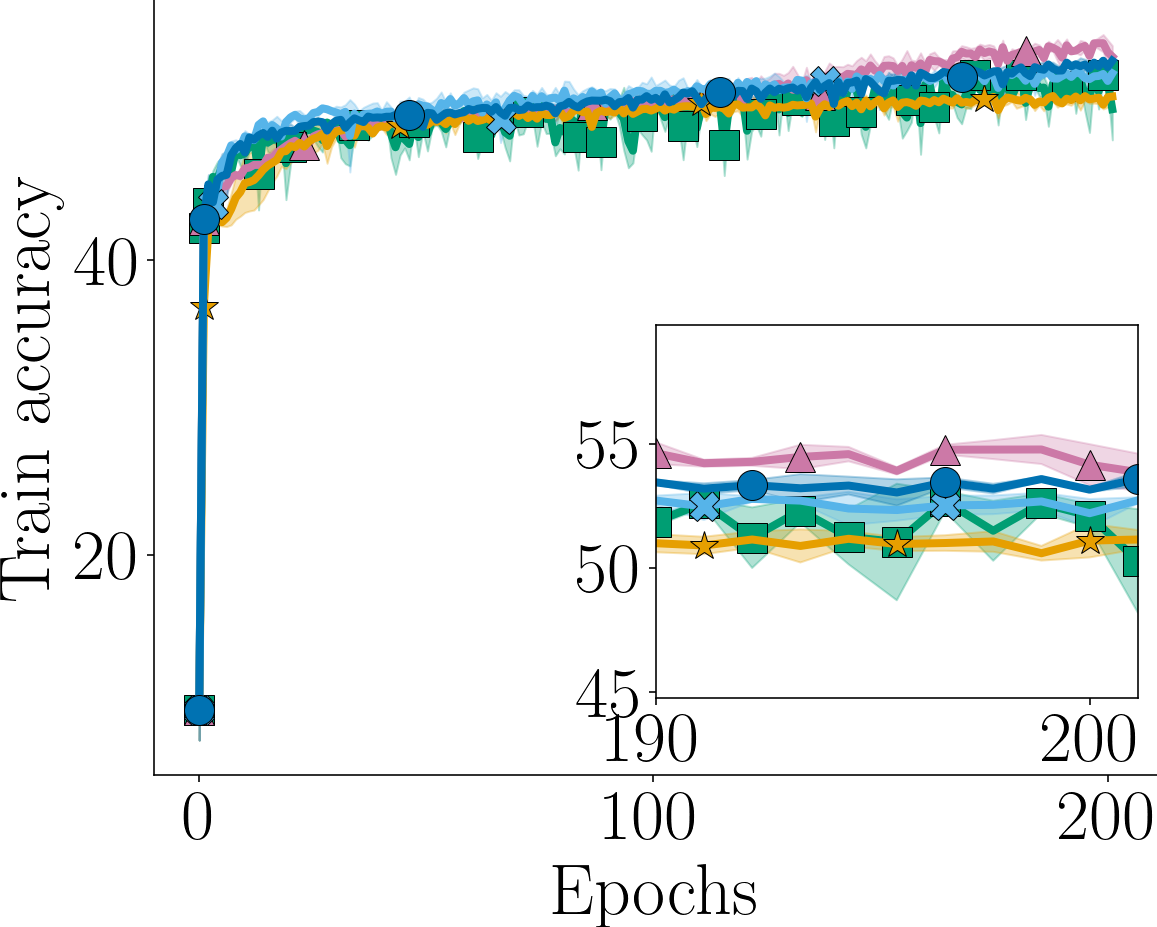} &
\includegraphics[width=0.3\textwidth]{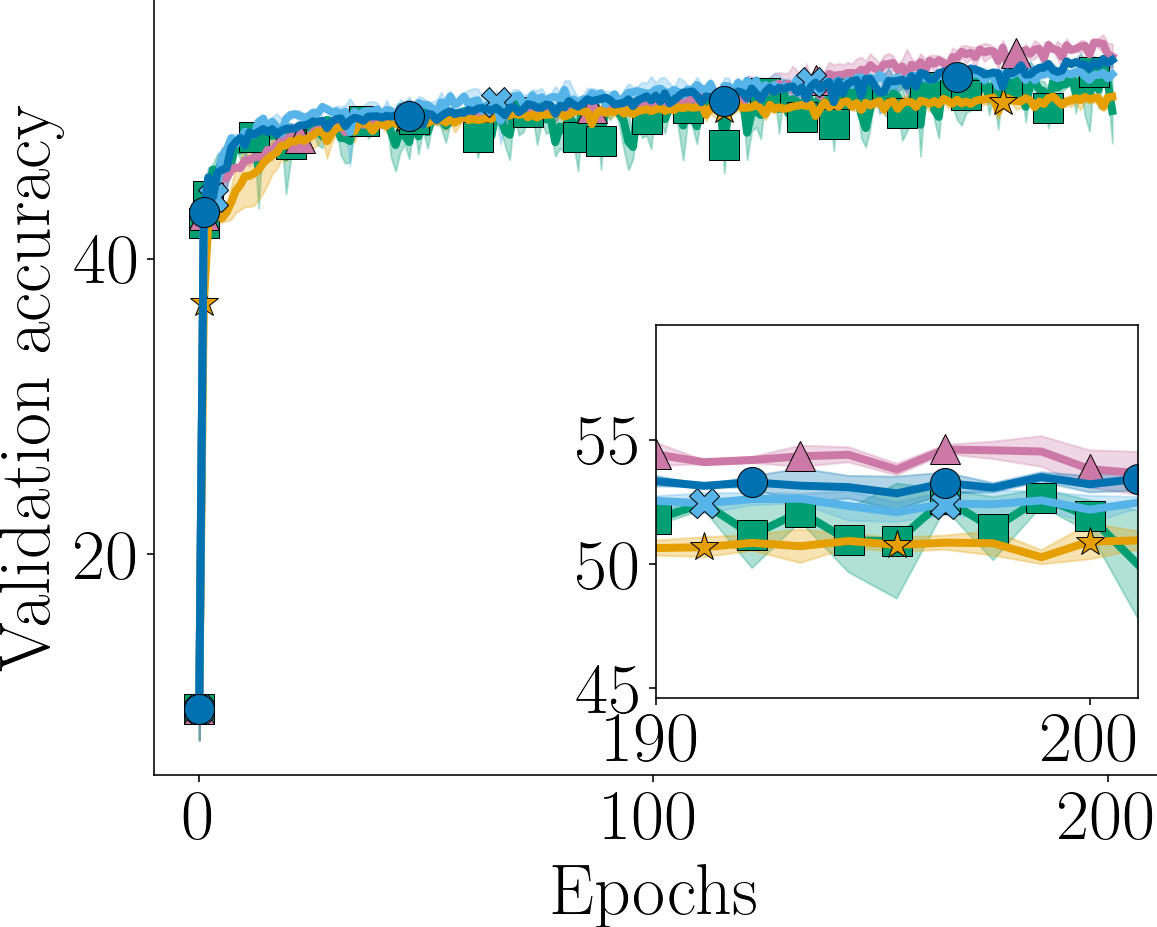} & \includegraphics[width=0.3\textwidth]{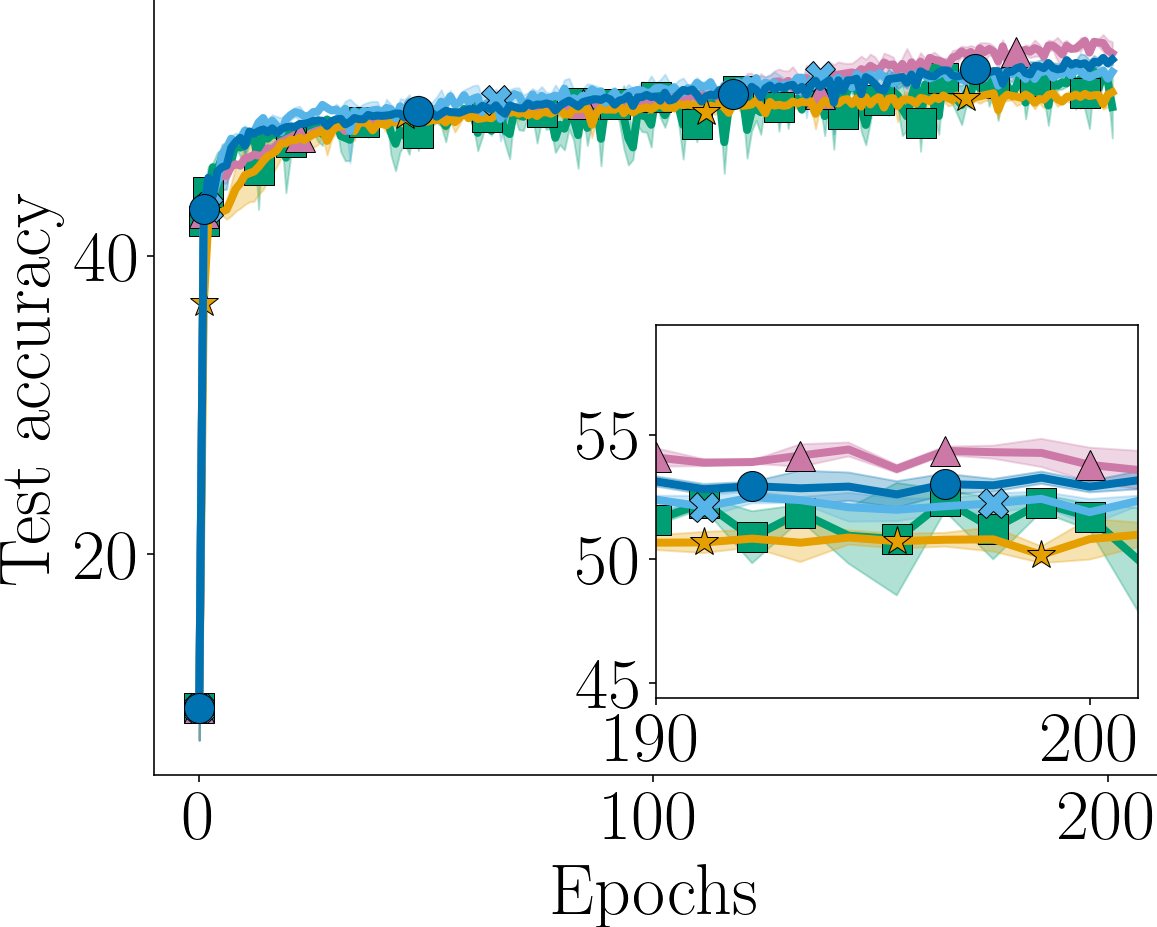} \\ \includegraphics[width=0.3\textwidth]{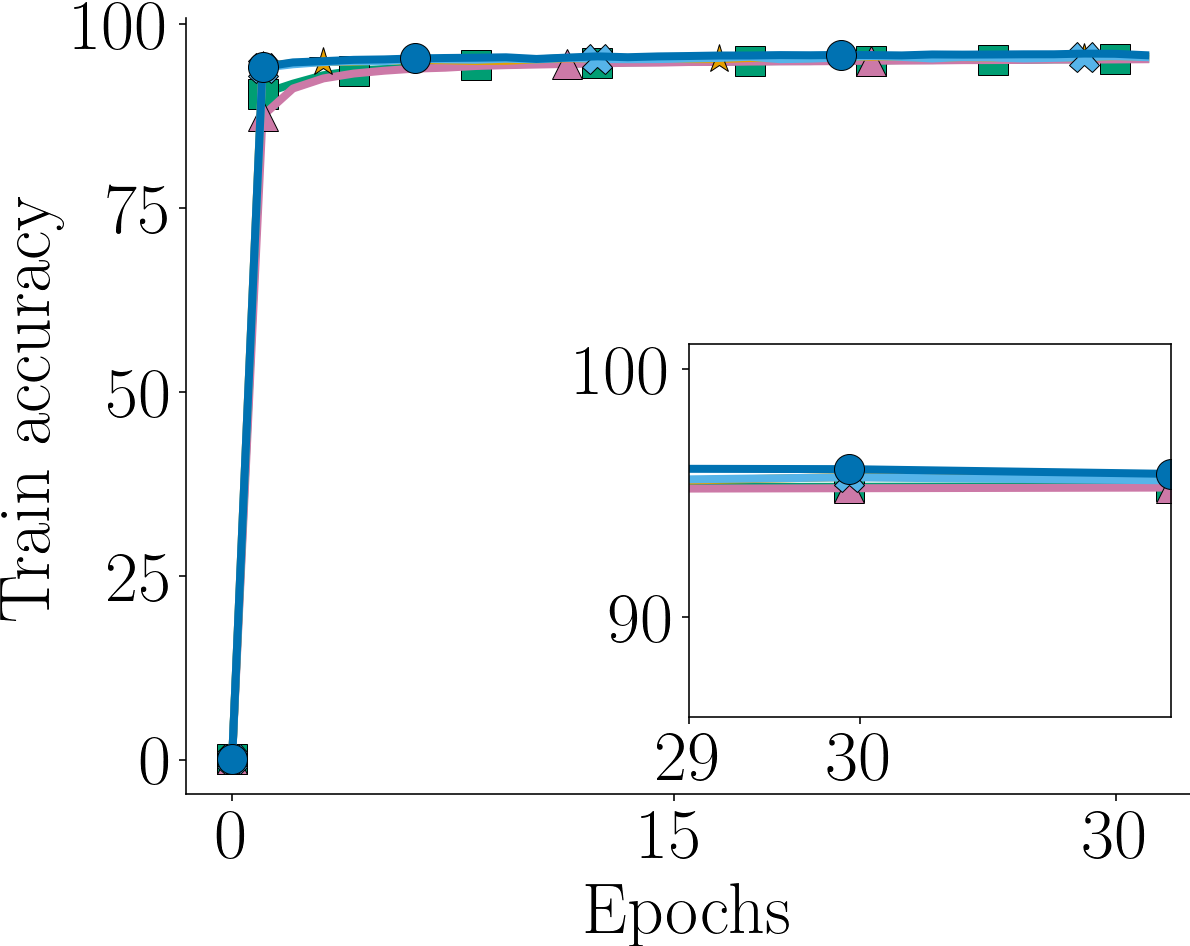} & \includegraphics[width=0.3\textwidth]{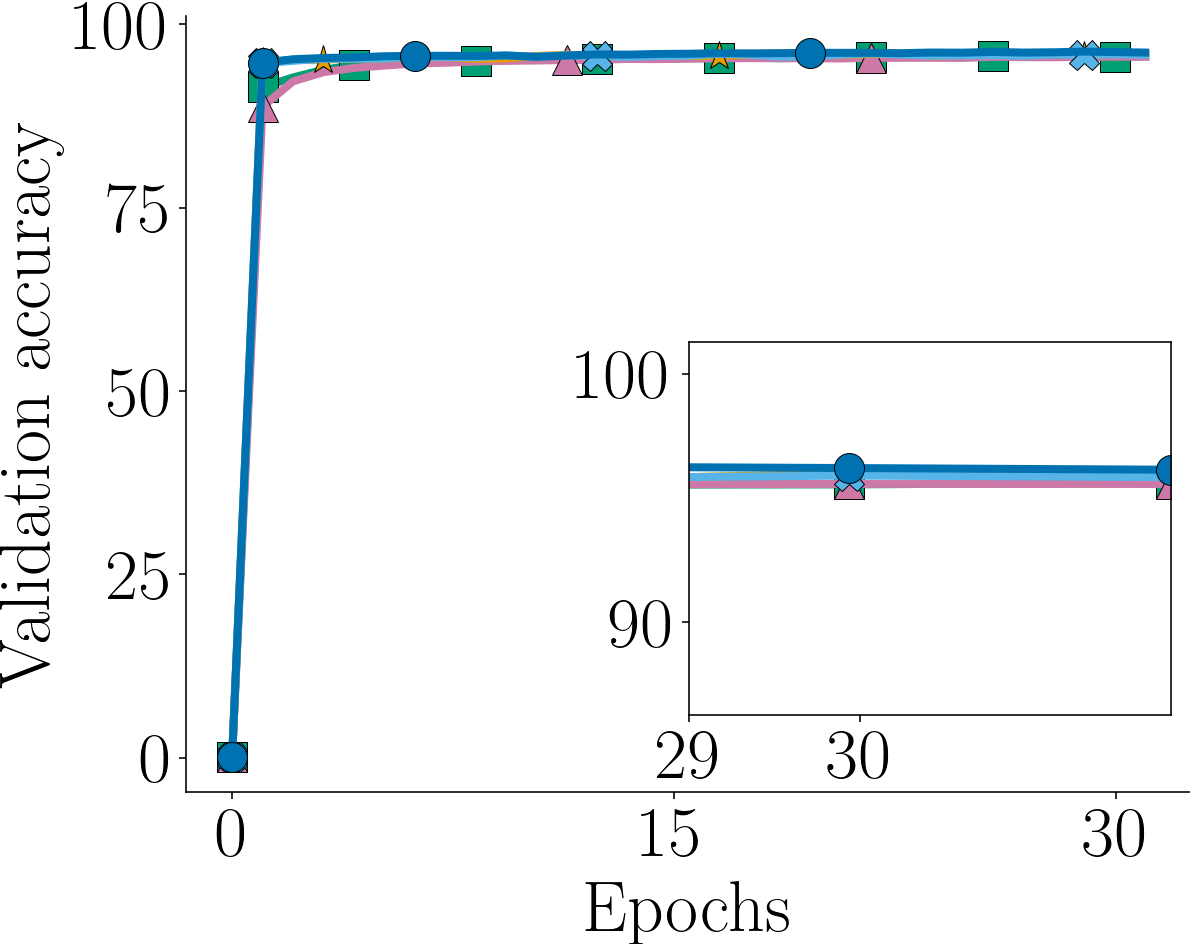} & \includegraphics[width=0.3\textwidth]{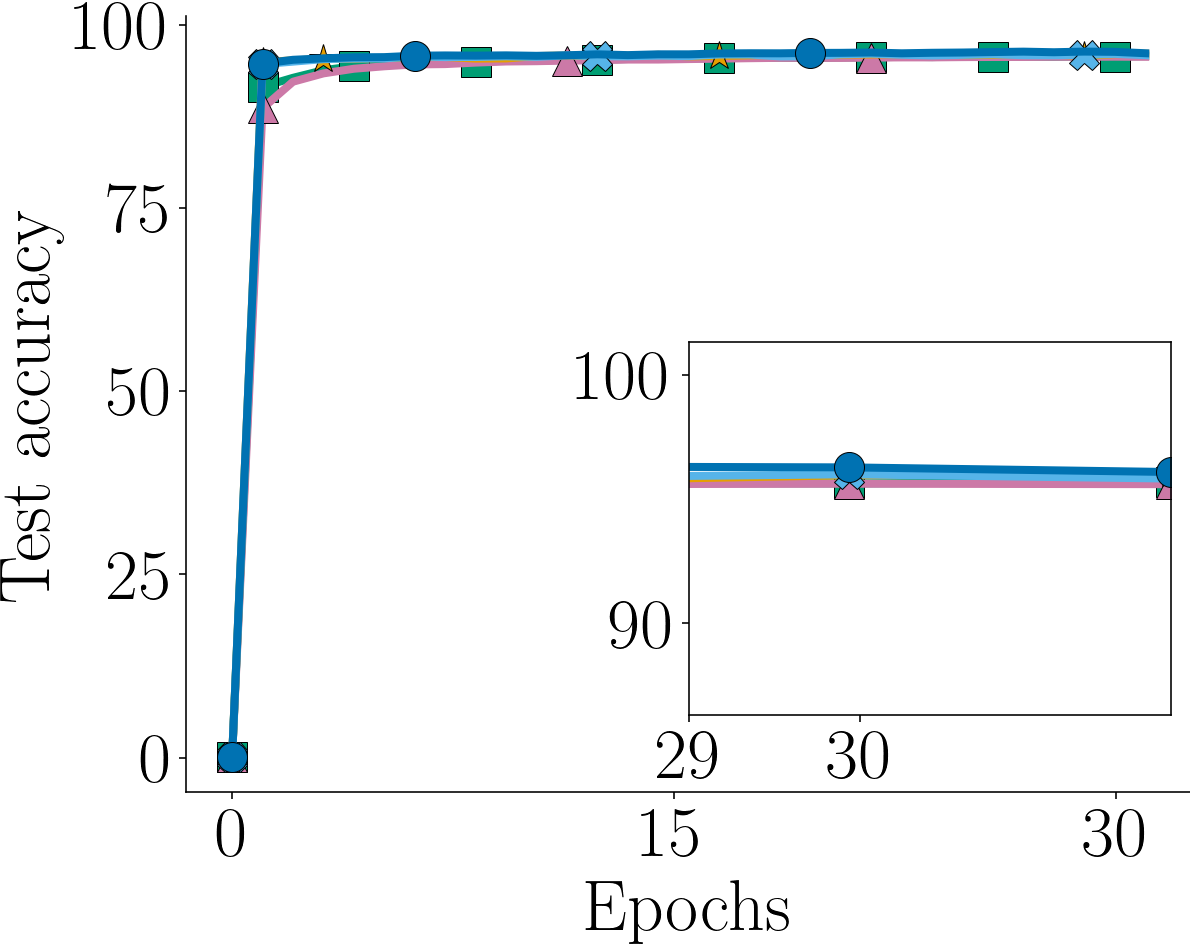} \\
\multicolumn{3}{c}{\includegraphics[width=0.7\textwidth]{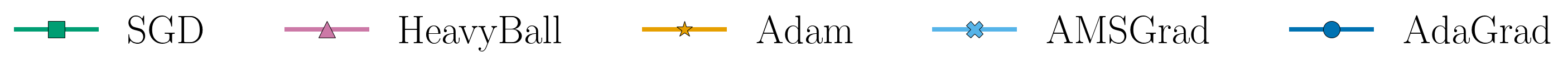}}
\end{tabular}
\caption{Training (first column), validation (second column), and testing (third column) curves for the datasets in Table~\ref{tab:datasets}. Top to bottom, each row corresponds to a dataset: Cora, CiteSeer, ogbn-arxiv, Flickr, Reddit. Here, ``HeavyBall'' refers to Heavy-Ball SGD}
\label{fig:accuracies}
\end{figure}

To test the effectiveness of the different optimizers for the node classification task~\eqref{eq:loss}, we train GCNs by Algorithm~\ref{alg:CV Training} on several benchmark datasets. We compare the five previously defined optimizers: SGD, Heavy-Ball SGD, Adam, AMSGrad, and AdaGrad. The training, validation, and testing results are shown in Figure \ref{fig:accuracies}. Plots were generated using matplotlib. Our code is a modified version of the implementation of VRGCN training\footnote{See: \url{https://github.com/THUDM/CogDL/tree/master/examples}.} which uses the CogDL library \cite{cen2021cogdl} and PyTorch \cite{paszke2019pytorch}. All experiments were performed on an AMD 7413 equipped with 256 GB RAM. \\
\begin{comment}

\end{comment}

%%%%%%%%%%%%%%%%%%%%%%%%%%%%%%%%%%%%%%%%%%%%%%
\subsection{Datasets}
The datasets we utilize are summarized in Table~\ref{tab:datasets}. The Cora, CiteSeer~\cite{sen2008collective}, and ogbn-arxiv~\cite{ogb-hu2020open} datasets are citation networks where nodes represent papers and edges represent citations; the Flickr~\cite{zeng2019graphsaint} dataset contains nodes which represent images and edges which represent shared properties; the Reddit~\cite{NS_GraphSAGE} dataset is a social network where nodes are Reddit posts and edges represent shared user comments. The node classification task assigns nodes to exactly one of the predetermined list of classes. For more information on how to access the benchmark datasets used in our code, see \url{https://github.com/RPI-OPT/CV-ADAM-GNN} .

\begin{table}[h]
\caption{Dataset summary. The last column denotes the number of classes for each dataset.}\label{tab:datasets}
\begin{tabular}{@{}lllll@{}}
\toprule
        Dataset & Nodes  & Features & Edges & Class\\
        \midrule
         ogbn-arxiv& 169,343& 128& 1,335,586& 40\\ 
 Flickr& 89,250& 500& 989,006& 7\\ 
 Reddit& 232,965   & 602    &  114,848,857   &  41\\
  Cora & 2,708  &1,433& 13,264 & 7 \\ 
  CiteSeer& 3,327& 3,703& 12,431& 6\\
\botrule
\end{tabular}
\end{table}

%%%%%%%%%%%%%%%%%%%%%%%%%%%%%%%%%%%%%%%%%%%%%%%%%%
\subsection{Hyperparameter Tuning}

We train $K=2$ layer GCNs for all numerical experiments. We set the momentum parameters to $\beta_1=0.9$ and $\beta_2=0.999$ for both Adam and AMSGrad, and the momentum parameter for Heavy-Ball SGD is set to $\beta_1=0.9$. For each dataset, the hidden dimension, weight decay, dropout, number of epochs, and number of runs for each experiment were fixed, see Table~\ref{tab:fixed hyperparams}.

The following hyperparameters were tuned in our numerical experiments: learning rate, sampled neighbors, and batch size. Sampled neighbors is the number of sampled neighbors at each layer \footnote{Note that in our implementation, the neighbor sampling is performed with replacement, with duplicates removed. Therefore the set number of sampled neighbors serves as an upper bound for the number of sampled neighbors.}, corresponding to $D^{(k)}$, which remains constant for all layers. 

Batch size is the size of the batch of sampled training nodes $\mathcal{V}_t$ at each iteration.
We used grid-search to test the hyperparameter values given in Table~\ref{tab:hyperparam ranges}. Due to the large size of the Reddit dataset, fewer hyperparameter experiments were performed.
The specific hyperparameter settings for each plot, as well as different hyperparameter values used in the experiments, can be found in Appendix~\ref{appendixB1}.

\begin{table}[h]
\caption{Hyperparameter values for grid search tuning for each dataset.}\label{tab:hyperparam ranges}%
\begin{tabular}{@{}llll@{}}
\toprule
Dataset & Learning rate  & Sampled neighbors  & Batch size\\
\midrule
ogbn-arxiv  & \{.001,.005,.01 \}  & \{2,5 \}  &  \{1000,2048,5000 \} \\
Flickr  &\{.1,.5,.8 \}   & \{2,5 \}  &  \{1000,2000,5000 \} \\
Reddit  & \{.01,.05 \}  & \{2 \}  &  \{1000 \} \\
Cora  &\{.001,.01,.05\}   &\{2,5 \}   & \{10,20,50 \}  \\
CiteSeer  &\{.001,.01,.05\}   &\{2,5 \}   & \{ 10,20,50\}  \\
\botrule
\end{tabular}
\end{table}

%%%%%%%%%%%%%%%%%%%%%%%%%%%%%%%%%%%%%
\begin{table}[h]
\caption{Fixed hyperparameter settings for each dataset.}\label{tab:fixed hyperparams}%
\begin{tabular}{@{}llllll@{}}
\toprule
Dataset& Hidden Dimension  & Weight Decay  & Dropout&Epochs&Runs\\
\midrule
ogbn-arxiv  &256   &.00001   &0  & 150&3\\
Flickr  & 256  &0   &.2  &200&3 \\
Reddit  &128   &0   & 0 & 30&1\\
Cora  &32   &.0005   & .5 &100 &3\\
CiteSeer  &32   &.0005   &.5  &100 &3 \\
\botrule
\end{tabular}
\end{table}

%%%%%%%%%%%%%%%%%%%%%%%%%%%%%%%%%%%%%%%%%%%%%%%%%%%
\subsection{Results}

The hyperparameters used for each curve in the train, test, and validation accuracy plots is the set of hyperparameters that produced the highest maximum test accuracy over an average of 3 runs (1 run for Reddit) for that particular optimizer and dataset. The corresponding plots of training, validation, and test accuracy are shown in Figure~\ref{fig:accuracies} and the corresponding maximum test accuracies are shown in Table \ref{tab:test_accs}. The highest maximum average test accuracy for each dataset is written in bold while the second highest maximum average test accuracy is underlined.
The Heavy-Ball SGD algorithm gave the highest test accuracy for the Flickr, Cora, and CiteSeer datasets, which contain the least number of nodes.
For the two larger datasets, the highest test accuracy for ogbn-arxiv was achieved with the AMSGrad optimizer while Adam achieved the second highest accuracy; for Reddit the AdaGrad optimizer achieved the highest test accuracy while AMSGrad achieved the second highest accuracy. In all cases, we find the addition of momentum or adaptivity greatly boosts performance over vanilla SGD.

\begin{table}[h]
\caption{Maximum average test accuracy (\%) from 3 independent trials (1 trial for Reddit). \textbf{Bold} indicates the best result for a given dataset while \underline{underlined} indicates the second best.}\label{tab:test_accs}%
\begin{tabular}{@{}llllll@{}}
        \toprule
        Dataset&Adam &Heavy-Ball& AMSGrad & AdaGrad & SGD\\
        \midrule
        ogbn-arxiv&\underline{74.11} &70.40 &\textbf{74.24} &72.47 &70.19 \\ 
    Flickr&50.97 & \textbf{54.41}& 52.49&\underline{53.27} &52.35\\
    Reddit&96.01 &95.62 &\underline{96.03} &\textbf{96.32} &95.70 \\ 
    Cora&80.20 &\textbf{81.10} &80.10 & 79.47&\underline{80.73}\\
    CiteSeer&69.00 &\textbf{71.07} &69.07 &68.40 &\underline{70.97}\\
        \bottomrule
    \end{tabular}
\end{table}

%%%%%%%%%%%%%%%%%%%%%%%%%%%%%%%%%%%%%%%%%%%%%%%%%%%%
\section{Conclusion}\label{sec:conclusion}

We have presented novel methods for training GCNs that combine the CVE technique and various Adam-type optimizers. We prove a general result of CVE-based Adam-type updates for solving  nonconvex GCN training problems. Unlike previous works on stochastic gradient-type methods, our result does not rely on the assumption of  unbiasedness of the used stochastic gradients, thanks to the CVE technique. We prove optimal convergence rates for a few specific settings of the Adam-type update: AMSGrad, AdaGrad, and Heavy-Ball SGD. Finally, we provide numerical results that compare the performance of different Adam-type optimizers for our method on training GCNs by using five benchmark datasets. The results demonstrate the superiority of the momentum or adaptive methods over the classic SGD on training GCNs, in particular on large graph datasets.

\section*{Acknowledgements} This work is partly supported by NSF grant DMS-2406896.
% \section*{Declarations}
% \subsection{Funding}
% This work was supported by 

% \subsection{Competing Interests}
% Not applicable.

% \subsection{Ethics approval and consent to participate}
% Not applicable.

% \subsection{Consent for Publication}

% \subsection{Data Availability}

% \subsection{Materials Availability}
% Not applicable.

% \subsection{Code Availability}
% Code used to run the numerical experiments is available at https://github.com/molly-noel/CV-Adam-GNN

% \subsection{Author Contribution}

\newpage

\begin{appendices}

\section{Proofs of Lemmas and Theorems}\label{appendixA1}
In this section, we provide details proofs of all theoretical results.

%%%%%%%%%%%%%%%%%%%%%%%%%%%%%%%%%%%%%%%%%%%%%%%%%%%%%%%%%%%%%
\subsection{Proof of Theorem~\ref{thm:main_result}}
We first establish several lemmas, which will be used to prove Theorem~\ref{thm:main_result}.

\begin{lemma}\label{lem:tilde-w-diff}
Let $\vW_{1}=\vW_0$. Define the sequence $\xt$ as follows:
\begin{align}
    \xt=\wt+\bfracConst(\vW_t-\vW_{t-1}),\;\;\; \forall t\geq 1. \label{eq:9}
\end{align}
\text{Then}
\begin{align*}
\xtn-\xt&=-\bfracConst \bigg(\frac{\at}{\sv}-\frac{\atp}{\svp} \bigg)\odot \mtp-\at\gt/\sv, \;\;\; \forall t\geq 1.
\end{align*}
\end{lemma}

\begin{proof}

For $t\geq 1$, using the update defined in Algorithm~\ref{alg:CV Training}, we have
    \begin{align}
        \wtn-\wt&=-\at\mt/\sv \notag \\
        &=-\at(\beta_1\mtp+(1-\beta_1)\gt)/\sv \notag\\
        &=-\at\beta_1\mtp/\sv-\at(1-\beta_1)\gt/\sv.\label{eq:l1_1}
    \end{align}
Again by the update rule, $-\alpha \vM_{t-1}=\sqrt{\vV_{t-1}}(\vW_t-\vW_{t-1})$. Therefore, \eqref{eq:l1_1} can be rewritten as 
    \begin{align}
        \wtn-\wt&=\beta_1\frac{\svp}{\sv}\odot(\wt-\wtp)-\at(1-\beta_1)\gt/\sv \notag\\
        &=\beta_1(\wt-\wtp)+\beta_1\bigg(\frac{\svp}{\sv}-1\bigg)\odot (\wt-\wtp)-\at(1-\beta_1)\gt/\sv \notag\\
        &=\beta_1(\wt-\wtp)+\beta_1\bigg(\frac{\svp}{\sv}-1\bigg)\odot(-\atp\mtp/\svp)-\at(1-\beta_1)\gt/\sv \notag\\
        &=\beta_1(\wt-\wtp)-\beta_1\bigg(\frac{\at}{\sv}-\frac{\atp}{\svp}\bigg)\odot \mtp-\at(1-\beta_1)\gt/\sv.
        \label{eq:10}
    \end{align}
Using the fact that $\wtn-\wt=(1-\beta_1)\wtn+\beta_1(\wtn-\wt)-(1-\beta_1)\wt$ and \eqref{eq:10}, we have
\begin{align}
    (1-\beta_1)\wtn+\beta_1(\wtn-\wt)&=\wtn-\wt+(1-\beta_1)\wt \notag\\
    &=(1-\beta_1)\wt+\beta_1(\wt-\wtp) \notag\\
    &-\beta_1\bigg( \frac{\at}{\sv}-\frac{\atp}{\svp} \bigg)\odot \mtp -\at(1-\beta_1)\gt/\sv \;\; \label{eq:l1_2}.
\end{align}
Dividing both sides of \eqref{eq:l1_2} by $(1-\beta_1)$ gives
\begin{align}
    \wtn+\bfracConst(\wtn-\wt)&=\wt+\bfracConst(\wt-\wtp) \notag\\
    &-\bfracConst\bigg( \frac{\at}{\sv}-\frac{\atp}{\svp} \bigg)\odot \mtp -\at\gt/\sv \;\;.\label{eq:11}
\end{align}
By the definition of $\xt$ in \eqref{eq:9},  
then \eqref{eq:11} can be written as
\begin{align*}
    \xtn&=\xt-\bfracConst \bigg( \frac{\at}{\sv}-\frac{\atp}{\svp} \bigg)\odot \mtp -\at\gt/\sv \;\;\; ,\forall t\geq 1.
\end{align*}
This completes the proof.
\end{proof}
%%%%%%%%%%%%%%%%%%%%%%%%%%%%%%%%%%%%%%%%%%%%%%%%%%%%%%%%%%%%

%%%%%%%%%%%%%%%%%%%%%%%%%%%%%%%%%%%%%%%%%%%%%%%%%%%%%%%%%%%
\begin{lemma}\label{lem:bound-obj-chg}
Let $\xt$ be defined in \eqref{eq:9}, then
    \begin{align}
            \EE[F(\xtn)-F(\widetilde{\vW}_1)]\leq \sum_{i=1}^4 \mathrm{Term}_i \label{eq:13}
    \end{align}
where
\begin{align}
    \mathrm{Term}_1&=-\EE\bigg[\sumi \bigg\langle \gfx ,\bfracConst \bigg(\frac{\ai}{\svi}-\frac{\aip}{\svip} \bigg)\odot \mip  \bigg\rangle  \bigg]         \label{eq:14} \\
    \mathrm{Term}_2&= -\EE\bigg[\sumi \ai\langle \gfx ,\gi/\svi \rangle  \bigg]       \label{eq:15} \\
    \mathrm{Term}_3&=  \EE \left[\sumi \rho \left\lVert\bfracConst \bigg( \frac{\ai}{\svi}-\frac{\aip}{\svip}\bigg)\odot \mip \right\rVert^2 \right]     \label{eq:18} \\
    \mathrm{Term}_4&=    \EE \bigg[\sumi \rho \|\ai \gi/\svi\|^2\bigg]    \label{eq:19} 
\end{align}
\end{lemma}

\begin{proof}
    By the Lipschitz continuity of $\nabla F$, we have:
    \begin{align}
        F(\widetilde{\vW}_{t+1})\leq  F(\widetilde{\vW}_t)+\langle \nabla F(\widetilde{\vW}_t),\Delta\widetilde{\vW}_t \rangle + \frac{\rho}{2}\|\Delta\widetilde{\vW}_t \|^2 \label{eq:20}
    \end{align}
    where $\Delta\widetilde{\vW}_t=\xtn-\xt$. Using Lemma~\ref{lem:tilde-w-diff}, we have
    \begin{align}
        \Delta\widetilde{\vW}_t=-\bfracConst \bigg(\frac{\at}{\sv}-\frac{\atp}{\svp}  \bigg)\odot \mtp  %\notag 
        -\at\gt/\sv  \;\;\; \forall t\geq 1 .\label{eq:21}
    \end{align}
Combining \eqref{eq:20} and \eqref{eq:21}, we have
\begin{align}
    \EE[F(\xtn)-F(\widetilde{\vW}_1)]&=\EE\bigg[\sumi F(\xin)-F(\xix)\bigg] \notag \\
    &\leq \EE\bigg[\sumi\langle \nabla F(\widetilde{\vW}_i),\Delta\widetilde{\vW}_i \rangle + \frac{\rho}{2}\|\Delta\widetilde{\vW}_i \|^2\bigg] \notag \\
    &=-\EE \bigg[\sumi \bigg\langle \gfx, \bfracConst \bigg(\frac{\ai}{
    \svi}-\frac{\aip}{\svip}  \bigg)\odot \mip \bigg\rangle   \bigg] \notag \\
    &- \EE \bigg[\sumi \ai \langle \gfx, \gi/\svi \rangle   \bigg] \notag 
    +\EE \bigg[\sumi \frac{\rho}{2} \|\Delta\widetilde{\vW}_i \|^2 \bigg] \notag \\
    &=\mathrm{Term}_1+\mathrm{Term}_2+ \EE \bigg[\sumi \frac{\rho}{2} \|\Delta\widetilde{\vW}_i \|^2  \bigg] \;\;.\label{eq:22}
\end{align}
Using $\|a+b \|^2 \leq 2 \| a\|^2+2 \| b\|^2$ and \eqref{eq:21}, we obtain
\begin{align*}
    \EE \bigg[\sumi \frac{\rho}{2} \|\Delta\widetilde{\vW}_i \|^2\bigg]= &\frac{\rho}{2}\EE \bigg[\sumi\left\lVert  -\bfracConst\bigg(\frac{\ai}{\svi}-\frac{\aip}{\svip}  \bigg)\odot \mip -\frac{\ai\gi}{\svi}  \right\rVert^2 \bigg]\\
    \leq & \frac{\rho}{2}\EE \bigg[\sumi 2\left\lVert \bfracConst\bigg(\frac{\ai}{\svi}-\frac{\aip}{\svip}  \bigg)\odot \mip \right\rVert^2 \bigg]\\
    &+\frac{\rho}{2}\EE \bigg[\sumi 2\|\ai\gi/\svi \|^2 \bigg]\\
    &=\mathrm{Term}_3+\mathrm{Term}_4.
\end{align*}
Substituting this inequality into \eqref{eq:22} results in \eqref{eq:13}.
\end{proof}

%%%%%%%%%%%%%%%%%%%%%%%%%%%%%%%%%%%%%%%%%%%%%%%%%%%%%%%%%%%%%%%%%%%%%%%
\begin{lemma}\label{lem:bd-term1}
Under Assumptions~\ref{assum:smooth}-\ref{assum:bound-grad}, $\mathrm{Term}_1$ in \eqref{eq:14} is bounded as:
    \begin{align*}
    \mathrm{Term}_1
    &\leq \xu{H_{\infty}^2} \frac{\beta_1}{1-\beta_1} \EE \bigg[\sum_{i=2}^t \sum_{j=1}^d \bigg|\bigg(\frac{\ai}{\svi}-\frac{\aip}{\svip}\bigg)_j\bigg|\bigg].
    \end{align*}
\end{lemma}
\begin{proof}
    Since $\|\gt \|\leq H_F$, from the update rule for $\mt$ in Algorithm~\ref{alg:CV Training}, it follows that $\|\mt\|\leq H_F$, which is proved by induction as follows.
    Since $\vM_0=0$, $\| \vM_0\|\leq H_F$.
        By the update rule for $\mt$ in Algorithm~\ref{alg:CV Training}, we have $\mt=\beta_1\mtp+(1-\beta_1)\gt$. Suppose $\|\mtp \|\leq H_F$. Then:
        \begin{align}\label{eq:bd-mt-F}
            \|\mt \|&= \| \beta_1\mtp+(1-\beta_1)\gt \|\leq \beta_1\|\mtp \|+(1-\beta_1)\|\gt\|\leq H_F .
        \end{align}
  
     Similarly, one can show 
     \begin{equation}\label{eq:bd-mt-inf}
     \| \mt\|_{\infty}\leq H_{\infty}.
     \end{equation}
     Then we bound $\mathrm{Term}_1$ as follows:
    \begin{align*}
       \mathrm{Term}_1&=-\EE\bigg[\sumi \bigg\langle \gfx ,\bfracConst \bigg(\frac{\ai}{\svi}-\frac{\aip}{\svip} \bigg)\odot \mip  \bigg\rangle  \bigg]  \\
       &=-\EE\bigg[\sum_{i=2}^t \bigg\langle \gfx ,\bfracConst \bigg(\frac{\ai}{\svi}-\frac{\aip}{\svip} \bigg)\odot \mip  \bigg\rangle  \bigg]  \\
       &\leq \frac{\beta_1}{1-\beta_1} \EE\bigg[\sum_{i=2}^t \| \gfx \|_\infty \bigg\| \bigg(\frac{\ai}{\svi}-\frac{\aip}{\svip} \bigg)\odot \mip \bigg\|_1  \bigg]  \\
       &\leq H_{\infty}^2\frac{\beta_1}{1-\beta_1}\EE\bigg[\sum_{i=2}^t \sum_{j=1}^d \bigg|\bigg(\frac{\ai}{\svi}-\frac{\aip}{\svip}\bigg)_j\bigg| \bigg]\\
    \end{align*}
where the second equality follows from $\vM_0=0$, the first inequality is by the Cauchy-Schwarz inequality, and the last inequality is from \eqref{eq:bd-mt-inf}.    
\end{proof}

%%%%%%%%%%%%%%%%%%%%%%%%%%%%%%%%%%%%%%%%%%%%%%%%%%%%%%%%%%
\begin{lemma}\label{lem:bd-term3}
    Under Assumptions~\ref{assum:smooth}-\ref{assum:bound-grad}, $\mathrm{Term}_3$ in \eqref{eq:18} is bounded as:
    \begin{align*}
        \mathrm{Term}_3\leq \rho\bigg(\frac{\beta_1}{1-\beta_1} \bigg)^2H_{\infty}^2\EE \bigg[\sum_{i=2}^t\sum_{j=1}^d \bigg(\frac{\ai}{\svi}-\frac{\aip}{\svip} \bigg)_j^2   \bigg]
    \end{align*}
\end{lemma}

\begin{proof}
It holds that
    \begin{align*}
        \mathrm{Term}_3&=\EE \bigg[\sumi  \rho\left\lVert\bfracConst \bigg( \frac{\ai}{\svi}-\frac{\aip}{\svip}\bigg)\odot \mip \right\rVert^2 \bigg]  \\
        &= \rho \EE \bigg[\sum_{i=2}^t \bigg( \frac{\beta_1}{1-\beta_1} \bigg)^2 \sum_{j=1}^d\bigg( \frac{\ai}{\svi}-\frac{\aip}{\svip}\bigg)^2_j (\mip)_j^2  \bigg]  \\
        &\leq\rho \bigg(\frac{\beta_1}{1-\beta_1} \bigg)^2H_{\infty}^2\EE \bigg[\sum_{i=2}^t\sum_{j=1}^d \bigg(\frac{\ai}{\svi}-\frac{\aip}{\svip} \bigg)_j^2   \bigg]
    \end{align*}
where the second equality follows from $\vM_0=0$, and the inequality holds by \eqref{eq:bd-mt-inf}.    
\end{proof}
%%%%%%%%%%%%%%%%%%%%%%%%%%%%%%%%%%%%%%%%%%%%%%%%
\begin{lemma}\label{lem:bd-term4}
Under Assumptions~\ref{assum:smooth}-\ref{assum:bound-grad}, $\mathrm{Term}_4$ in \eqref{eq:19} is bounded as:
\begin{align}
    \mathrm{Term}_4\leq \frac{\rho\alpha^2}{\nu_{\mathrm{min}}^2}H_F^2 t .
\end{align}
\end{lemma}
\begin{proof}
    This term can simply be bounded by using Assumptions~\ref{assum:low-bd} and \ref{assum:bound-grad} as follows 
    \begin{align*}
     \mathrm{Term}_4=\EE \bigg[\sum_{i=1}^t \rho\|\alpha \vG_i/\sqrt{\vV_i} \|^2   \bigg] \leq \frac{\rho\alpha^2}{\nu_{\mathrm{min}}^2} \EE \bigg[\sum_{i=1}^t\| \vG_i \|^2\bigg]\leq  \frac{\rho\alpha^2}{\nu_{\mathrm{min}}^2}H_F^2 t, 
    \end{align*}
 which gives the desired result.   
\end{proof}

To bound $\mathrm{Term}_2$ in \eqref{eq:15}, we need the following lemma.
%%%%%%%%%%%%%%%%%%%%%%%%%%%%%%%%%%%%%%%%%%%%%%%%%%%%%%%%%%%%%%%%%%%%
%%%%%%%%%%%%%%%%%%%%%%%%%%%%%%%%%%%%%%%%%%%%%%%%%%%%%
\begin{lemma}\label{lem:delta_bound}
Under Assumption~\ref{assum:bound-grad}, let $\bm{\delta}_i=\vG_i-\nabla F(\vW_i)$, then
$$\|\EE_{\xi}[\bm{\delta}_i] \|_{\infty}\leq C\frac{\alpha}{\nu_{\mathrm{min}}}H_{\infty} ,$$ 
where $C$ is a universal constant, and 
    $\xi$ is the random variable accounting for all the randomness in $\widehat{\vP}^{(k)},k=1,\ldots,K$,  $\cV_t$,
    and neighbor sampling at each layer. 
\end{lemma}
\begin{proof}
     This can be done using Lemma 2 from \cite{VRGCN}, which states:
$$\|\EE_{\xi}[\bm{\delta}_i] \|_{\infty}\leq RQ   $$ where $R$ is a universal constant and $Q$ is defined as follows: $$\|\vW_i-\vW_j \|_{\infty}\leq Q\;\;\forall i,j.$$ Let $J$ be the number of iterations per epoch of training. To find the value of $Q$ for this problem setting, we follow the proof of Theorem 2 in \cite{VRGCN} and have
\begin{align*}
    \max_{i-KJ\leq j, k\leq i} \|\vW_j-\vW_{k} \|_{\infty}&\leq\sum_{j=i-KJ}^{i-1}\|\vW_j-\vW_{j+1} \|_{\infty}\\
    &=\sum_{j=i-KJ}^{i-1}\|\alpha \vM_j/\sqrt{\vV_j} \|_{\infty}\\
    &\leq \sum_{j=i-KJ}^{i-1}\frac{\alpha}{\nu_{\mathrm{min}}}\|\vM_j \|_{\infty}\\
     &\leq \sum_{j=i-KJ}^{i-1}\frac{\alpha}{\nu_{\mathrm{min}}}H_{\infty}\\
     &=KJ\frac{\alpha}{\nu_{\mathrm{min}}}H_{\infty}.
\end{align*}
Using this value for $Q$, we have 
$$\|\EE_{\xi}[\bm{\delta}_i] \|_{\infty}\leq RKT\frac{\alpha}{\nu_{\mathrm{min}}}H_{\infty}=C\frac{\alpha}{\nu_{\mathrm{min}}}H_{\infty} $$ 
where $C=RKJ$.

%%%%%%%%%%%%%%%%%%%%%%%%%%%%%%%%%%%%%%%%%%%%%%%%%%%%%
\end{proof}

%%%%%%%%%%%%%%%%%%%%%%%%%%%%%%%%%%%%%%%%%%%%%%%%%%%%%%%%%%%%%%%%%%%%
\begin{lemma}\label{lem:bd-term2}
    Under Assumptions~\ref{assum:smooth}-\ref{assum:bound-grad}, $\mathrm{Term}_2$ in \eqref{eq:15} can be bounded as:
    \begin{align}
     \mathrm{Term}_2
     &\leq\frac{\alpha^2}{2\nu_{\mathrm{min}}^2}H_F^2(t-1)\bigg[\rho^2\xu{\frac{\beta_1^2}{(1-\beta_1)^2}}+1 \bigg] +\xu{2H^2_{\infty}}\EE\bigg[\sum_{i=2}^t\sum_{j=1}^d  \bigg|\bigg( \frac{\ai}{(\svi)_j}-\frac{\aip}{(\svip)_j} \bigg)\bigg| \bigg] \notag\\
     &+\alpha^2 \frac{CH_1H_{\infty}}{\nu_{\mathrm{min}}^2}(t-1) +\alpha\frac{\xu{H_1H_{\infty}}}{\nu_{\mathrm{min}}} -\EE\bigg[\sumi \ai\langle \nabla F(\vW_i) ,\nabla F(\vW_i)/\svi \rangle  \bigg].  \label{eq:24}
    \end{align}
\end{lemma}

\begin{proof}
    From \eqref{eq:9}, we have
    \begin{align}
        \widetilde{\vW}_i-\vW_i=\bfracConst(\vW_i-\vW_{i-1})=-\bfracConst\aip\mip/\svip \;\;.\label{eq:25}
    \end{align}
    From the definition of $\widetilde{\vW}_i$, we have $\widetilde{\vW}_1=\vW_1$. Therefore,
    \begin{align}
    \mathrm{Term}_2&= -\EE\bigg[\sumi \ai\langle \gfx ,\gi/\svi \rangle  \bigg] \notag\\
    &= -\EE\bigg[\sumi \ai\langle \nabla F(\vW_i) ,\gi/\svi \rangle  \bigg] -\EE\bigg[\sum_{i=2}^t \ai\langle \nabla F(\widetilde{\vW}_i)-\nabla F(\vW_i) ,\gi/\svi \rangle  \bigg] \;\;.\label{eq:26} 
    \end{align}
    Using the equation $\langle a,b \rangle \leq \frac{1}{2}(\| a\|^2+\|b\|^2)$ and the fact that the gradient of $F$ is  $\rho$-Lipschitz, the second term in the RHS of \eqref{eq:26} can be bounded as follows:
    \begin{align}
    & -\EE\bigg[\sum_{i=2}^t \ai\langle \nabla F(\widetilde{\vW}_i)-\nabla F(\vW_i) ,\gi/\svi \rangle  \bigg] \notag \\
    \leq &\, \EE\bigg[\sum_{i=2}^t  \frac{1}{2}\|\nabla F(\widetilde{\vW}_i)-\nabla F(\vW_i)\|^2+ \frac{1}{2}\|\ai\gi/\svi\|^2  \bigg] \notag\\
    \leq &\, \frac{\rho^2}{2} \EE\bigg[\sum_{i=2}^t \|\widetilde{\vW}_i-\vW_i\|^2\bigg]+ \frac{1}{2}\EE\bigg[\sum_{i=2}^t \|\ai\gi/\svi\|^2\bigg]\notag\\
    = &\, \frac{\rho^2}{2} \EE\bigg[\sum_{i=2}^t \left\lVert\bfracConst \aip\mip/\svip  \right\rVert^2\bigg] + \frac{1}{2}\EE\bigg[\sum_{i=2}^t \|\ai\gi/\svi\|^2\bigg] \label{eq:27}
    \end{align}

To bound the first term of \eqref{eq:27}, we use the bounds: $\|\vM_t\| \le H_F$ from \eqref{eq:bd-mt-F} and $(\vV_i)_j \ge \nu_{\mathrm{min}}^2 \; \forall j$ to have
\begin{align}
    \frac{\rho^2}{2} \EE\bigg[\sum_{i=2}^t \left\lVert\bfracConst \aip\mip/\svip  \right\rVert^2\bigg]& = 
    \frac{\rho^2\alpha^2}{2}\xu{\frac{\beta_1^2}{(1-\beta_1)^2}}\EE\bigg[\sum_{i=2}^t\left\lVert\mip/\svip  \right\rVert^2 \bigg]\notag \\
    &\leq \frac{\rho^2\alpha^2}{2\nu_{\mathrm{min}}^2}\xu{\frac{\beta_1^2}{(1-\beta_1)^2}} H_F^2(t-1)
\end{align}
Similarly, we can bound the second term of \eqref{eq:27} by 
\begin{align}
    \frac{1}{2}\EE\bigg[\sum_{i=2}^t \|\ai\gi/\svi\|^2\bigg]\leq \frac{\alpha^2}{2\nu_{\mathrm{min}}^2}H_F^2(t-1).
\end{align}
Putting these results together, the second term of \eqref{eq:26} can be bounded as follows:
 \begin{align}
     -\EE\bigg[\sum_{i=2}^t \ai\langle \nabla F(\widetilde{\vW}_i)-\nabla F(\vW_i) ,\gi/\svi \rangle  \bigg]&\leq \frac{\rho^2\alpha^2}{2\nu_{\mathrm{min}}^2}\frac{\beta_1^2}{(1-\beta_1)^2}H_F^2(t-1)+\frac{\alpha^2}{2\nu_{\mathrm{min}}^2}H_F^2(t-1) \notag \\
     &=\frac{\alpha^2}{2\nu_{\mathrm{min}}^2}H_F^2(t-1)\bigg[\rho^2\xu{\frac{\beta_1^2}{(1-\beta_1)^2}}+1 \bigg] \label{eq:t1}
 \end{align}

Now, to bound %the negative of 
the first term in \eqref{eq:26}, let $$\bm{\delta}_t:=\gt-\nabla F(\vW_t),$$
and thus 
\begin{align}
     \EE\bigg[\sumi \ai\langle \nabla F(\vW_i) ,\gi/\svi \rangle  \bigg]
     = & \EE\bigg[\sumi \ai\langle \nabla F(\vW_i) ,\nabla F(\vW_i)/\svi \rangle  \bigg] \notag \\
     &+\EE\bigg[\sumi \ai\langle \nabla F(\vW_i) ,\bm{\delta}_i/\svi \rangle  \bigg]. \label{eq:34}
\end{align}

The first term in the RHS of \eqref{eq:34} is nonegative and is the descent quantity to be bounded in the convergence proof. To bound the second term in the RHS of \eqref{eq:34}, we rewrite it as follows:
\begin{align}
    \EE\bigg[\sumi \ai\langle \nabla F(\vW_i) ,\bm{\delta}_i/\svi \rangle  \bigg]&=\EE\bigg[\sum_{i=2}^t \bigg\langle \nabla F(\vW_i) ,\bm{\delta}_i \odot\bigg( \frac{\ai}{\svi}-\frac{\aip}{\svip} \bigg) \bigg\rangle  \bigg] \label{eq:delta1} \\
    &+\EE\bigg[\sum_{i=2}^t \bigg\langle \nabla F(\vW_i) ,\bm{\delta}_i \odot\frac{\alpha}{\svip} \bigg\rangle  \bigg] \label{eq:delta2} \\
    &+\EE\bigg[ \alpha\langle \nabla F(\vW_1) ,\bm{\delta}_1/\sqrt{\vV_1} \rangle  \bigg]. \label{eq:delta3}
\end{align}

Below we upper bound the three terms in the RHS of the above equation. 
Since $\|\vG_t \|\leq H_F$ and $\|\nabla F(\vW_t) \|\leq H_F$ from Assumption~\ref{assum:bound-grad}, we have $\|\bm{\delta}_t \|\leq 2H_F$. Similarly, $\|\bm{\delta}\|_{\infty}\leq 2H_{\infty}$.
Hence, we bound \eqref{eq:delta1} as follows: 
\begin{align*}
 \EE\bigg[\sum_{i=2}^t \bigg\langle \nabla F(\vW_i) ,\bm{\delta}_i \odot\bigg( \frac{\ai}{\svi}-\frac{\aip}{\svip} \bigg) \bigg\rangle  \bigg]
 &\geq -\xu{2H^2_{\infty}}\EE\bigg[\sum_{i=2}^t\sum_{j=1}^d  \bigg|\bigg( \frac{\ai}{(\svi)_j}-\frac{\aip}{(\svip)_j} \bigg)\bigg| \bigg]. 
\end{align*}
To bound \eqref{eq:delta2}, we use the result from Lemma~\ref{lem:delta_bound} to bound $\|\EE_{\xi_i}[\bm{\delta}_i]\|$: 

\begin{align*}
\EE\bigg[\sum_{i=2}^t \aip\bigg\langle \nabla F(\vW_i) ,\bm{\delta}_i \odot\frac{1}{\svip} \bigg\rangle  \bigg]&=\EE\bigg[\sum_{i=2}^t \aip\bigg\langle \nabla F(\vW_i) ,\EE_{\xi_i}[\bm{\delta}_i] \odot\frac{1}{\svip} \bigg\rangle \bigg|\xi_1,\ldots,\xi_{i-1} \bigg]\\
&\geq -\EE\bigg[\sum_{i=2}^t \frac{\aip}{\nu_{\mathrm{min}}}\|\nabla F(\vW_i) \|_1 \| \EE_{\xi_i}[\bm{\delta}_i]\|_{\infty} \bigg|\xi_1,\ldots,\xi_{i-1}  \bigg]\\
&\geq -\EE\bigg[\sum_{i=2}^t \frac{\aip}{\nu_{\mathrm{min}}} H_1 \frac{C\alpha}{\nu_{\mathrm{min}}} H_{\infty} \bigg]\\
&= -\alpha^2 \frac{CH_1H_{\infty}}{\nu_{\mathrm{min}}^2}(t-1).
\end{align*}
To bound \eqref{eq:delta3}, we have:
\begin{align*}
\EE\bigg[ \alpha\langle \nabla F(\vW_1) ,\bm{\delta}_1/\sqrt{\vV_1} \rangle  \bigg]&\geq -\EE\bigg[ \frac{\alpha}{\nu_{\mathrm{min}}}\| \nabla F(\vW_1)\|_1 \|\bm{\delta}_1\|_{\infty}  \bigg] \geq -\alpha\frac{\xu{H_1H_{\infty}}}{\nu_{\mathrm{min}}}.
\end{align*}
Putting these together, we can then bound the second term in the RHS of \eqref{eq:34} by
\begin{align}
     \EE\bigg[\sumi \ai\langle \nabla F(\vW_i) ,\bm{\delta}_i/\svi \rangle  \bigg] &\geq -\xu{2H^2_{\infty}}\EE\bigg[\sum_{i=2}^t\sum_{j=1}^d  \bigg|\bigg( \frac{\ai}{(\svi)_j}-\frac{\aip}{(\svip)_j} \bigg)\bigg| \bigg] \notag\\
     &-\alpha^2 \frac{CH_1H_{\infty}}{\nu_{\mathrm{min}}^2}(t-1) -\alpha\frac{\xu{H_1H_{\infty}}}{\nu_{\mathrm{min}}}.\label{eq:36}
\end{align}

Then we can bound the first term in the RHS of \eqref{eq:26} by substituting \eqref{eq:36} into \eqref{eq:34}
\begin{align}
    -\EE\bigg[\sumi \ai\langle \nabla F(\vW_i) ,\gi/\svi \rangle  \bigg]
     &\leq  \xu{2H^2_{\infty}}\EE\bigg[\sum_{i=2}^t\sum_{j=1}^d  \bigg|\bigg( \frac{\ai}{(\svi)_j}-\frac{\aip}{(\svip)_j} \bigg)\bigg| \bigg] \notag\\
     &+\alpha^2 \frac{CH_1H_{\infty}}{\nu_{\mathrm{min}}^2}(t-1) +\alpha\frac{\xu{H_1H_{\infty}}}{\nu_{\mathrm{min}}} \notag\\ 
     &-\EE\bigg[\sumi \ai\langle \nabla F(\vW_i) ,\nabla F(\vW_i)/\svi \rangle  \bigg]. \label{eq:37}
\end{align}
Combining equations \eqref{eq:37} and \eqref{eq:t1} to gives the inequality in \eqref{eq:24}.
\end{proof}

%%%%%%%%%%%%%%%%%%%%%%%%%%%%%%%%%%%%%%%%%%%%%%%%%%%%%%%%%%%%%
%\subsection{Proof of Theorem 1}
Now we are ready to prove Theorem~\ref{thm:main_result}.
\begin{proof}[\textbf{Proof Of Theorem~\ref{thm:main_result}}]
Starting from the result in Lemma~\ref{lem:bound-obj-chg} and bounding $\mathrm{Term}_1, \mathrm{Term}_2, \mathrm{Term}_3, \mathrm{Term}_4$ by using the results from Lemmas~\ref{lem:bd-term1}--\ref{lem:bd-term2}, we have:

\begin{align*}
   &\, \EE[F(\xtn)-F(\widetilde{\vW}_1)]
   \\
    \leq &\, H_{\infty}^2 \frac{\beta_1}{1-\beta_1} \EE \bigg[\sum_{i=2}^t \sum_{j=1}^d \bigg|\bigg(\frac{\ai}{\svi}-\frac{\aip}{\svip}\bigg)_j\bigg|\bigg] \\ %Term1
%%%%%%%%%%%%%%%%%%%%%%%%%%%%%%%%%%%%%%%%%%%%%%%%%%%%%%%
    &+\rho\bigg(\frac{\beta_1}{1-\beta_1} \bigg)^2H_{\infty}^2\EE \bigg[\sum_{i=2}^t\sum_{j=1}^d \bigg(\frac{\ai}{\svi}-\frac{\aip}{\svip} \bigg)_j^2   \bigg] %Term3
%%%%%%%%%%%%%%%%%%%%%%%%%%%%%%%%%%%%%%%%%%%%%%%%%%%%%%
    +\frac{\rho\alpha^2}{\nu_{\mathrm{min}}^2}H_F^2t \\%Term4
%%%%%%%%%%%%%%%%%%%%%%%%%%%%%%%%%%%%%%%%%%%%%%%%%%%%%%%
    &+\frac{\alpha^2}{2\nu_{\mathrm{min}}^2}H_F^2(t-1)\bigg[\rho^2\frac{\beta_1^2}{(1-\beta_1)^2}+1 \bigg] \notag +2H^2_{\infty}\EE\bigg[\sum_{i=2}^t\sum_{j=1}^d  \bigg|\bigg( \frac{\ai}{(\svi)_j}-\frac{\aip}{(\svip)_j} \bigg)\bigg| \bigg] \notag\\
     &+\alpha^2 \frac{CH_1H_{\infty}}{\nu_{\mathrm{min}}^2}(t-1) +\alpha\frac{H_1H_{\infty}}{\nu_{\mathrm{min}}} \notag -\EE\bigg[\sumi \ai\langle \nabla F(\vW_i) ,\nabla F(\vW_i)/\svi \rangle  \bigg].%Term2
%%%%%%%%%%%%%%%%%%%%%%%%%%%%%%%%%%%%%%%%%%%%%%%%%%%%%%%
\end{align*}
    %%%%%%%%%%%%%%%%%%%%%%%%%%%%%%%%%%%
Rearranging terms in the above inequality gives %this result allows us to bound the desired descent quantity as follows:
\begin{align*}
&\, \EE\bigg[\sumi \ai\bigg\langle \nabla F(\vW_i) ,\frac{\nabla F(\vW_i)}{\svi} \bigg\rangle  \bigg]\\
\leq &\,
 (H_{\infty}^2\frac{\beta_1}{1-\beta_1}+2H_{\infty}^2) \EE \bigg[\sum_{i=2}^t \sum_{j=1}^d \bigg|\bigg(\frac{\ai}{\svi}-\frac{\aip}{\svip}\bigg)_j\bigg|\bigg]\\ %one-norm
%%%%%%%%%%%%%%%%%%%%%%%%%%%%%%%%%%%%%%%%%%%%%%%%%%%%%%%
&+\rho\bigg(\frac{\beta_1}{1-\beta_1} \bigg)^2H_{\infty}^2\EE \bigg[\sum_{i=2}^t\sum_{j=1}^d \bigg(\frac{\ai}{\svi}-\frac{\aip}{\svip} \bigg)_j^2   \bigg] %two-norm
%%%%%%%%%%%%%%%%%%%%%%%%%%%%%%%%%%%%%%%%%%%%%%%%
+\bigg[ \rho\frac{\alpha^2}{\nu_{\mathrm{min}}^2}H_F^2 \bigg]t\\% t term
%%%%%%%%%%%%%%%%%%%%%%%%%%%%%%%%%%%%%%%%%%%%%%%%%%%%
&+\bigg[\frac{\alpha^2}{2\nu_{\mathrm{min}}^2}H_F^2\bigg(\rho^2\frac{\beta_1^2}{(1-\beta_1)^2}+1 \bigg)+\frac{\alpha^2}{\nu_{\mathrm{min}}^2}CH_1H_{\infty}  \bigg](t-1)\\% t-1 term
%%%%%%%%%%%%%%%%%%%%%%%%%%%%%%%%%%%%%%%%%%%%%%%%%%%%%%%%%%%%
&+\alpha \frac{H_1H_{\infty}}{\nu_{\mathrm{min}}} + \EE[F(\widetilde{\vW}_1)-F(\xtn)].%\\
%%%%%%%%%%%%%%%%%%%%%%%%%%%%%%%%%%%%%%%%%%%%%%%%%%%
\end{align*}
By $F(\xtn) \ge F^*$, $\widetilde{\vW}_1={\vW}_1$, and the definitions of the constants $C_1,C_2,C_3$ and $C_4$ in \eqref{eq:C1}-\eqref{eq:C4}, we obtain the desired result and complete the proof.
\end{proof}

\section{Hyperparameter Tuning Numerical Results}\label{appendixB1}

%%%%%%%%%%%%%%%%%%%%%%%%%%%%%%%%%%%%
\begin{table}[h]
\caption{Hyperparameters that generate highest test accuracy for each optimizer on Cora}\label{tab:cora}%
\begin{tabular}{@{}llll@{}}
        \toprule
        Optimizer &Learning Rate &Sampled Neighbors & Batch Size
       \\
        \midrule
        Adam& .01 &5  &20    \\
        Heavy-Ball SGD&.05  &2  &50    \\
        AMSGrad&.01  &2  &10    \\
        AdaGrad &.01  &5  &10    \\
        SGD&.05  &2  &20    \\
        
        \bottomrule
    \end{tabular}
\end{table}

\begin{table}[h]
\caption{Hyperparameters that generate highest test accuracy for each optimizer on CiteSeer}\label{tab:citeseer}%
\begin{tabular}{@{}llll@{}}
        \toprule
        Optimizer &Learning Rate &Sampled Neighbors & Batch Size
       \\
        \midrule
        Adam&.01  &2  &10    \\
        Heavy-Ball SGD&.01  &2  &20    \\
        AMSGrad&.01  &5  &20    \\
        AdaGrad &.01  &5  &10    \\
        SGD&.05  &2  &10    \\
        
        \bottomrule
    \end{tabular}
\end{table}

\begin{table}[h]
\caption{Hyperparameters that generate highest test accuracy for each optimizer on ogbn-arxiv}\label{tab:ogbn-arxiv}%
\begin{tabular}{@{}llll@{}}
        \toprule
        Optimizer &Learning Rate &Sampled Neighbors & Batch Size
       \\
        \midrule
        Adam&.005  &2 &1000    \\
        Heavy-Ball SGD& .01 &2  &1000    \\
        AMSGrad&.005  &2  &1000    \\
        AdaGrad &.01  &2  &1000    \\
        SGD&.01  &2  &1000    \\
        
        \bottomrule
    \end{tabular}
\end{table}

\begin{table}[h]
\caption{Hyperparameters that generate highest test accuracy for each optimizer on Flickr}\label{tab:flickr}%
\begin{tabular}{@{}llll@{}}
        \toprule
        Optimizer &Learning Rate &Sampled Neighbors & Batch Size
       \\
        \midrule
        Adam& .1 &5  &5000    \\
        Heavy-Ball SGD&.1  &5  &1000    \\
        AMSGrad&.1  & 5 & 1000   \\
        AdaGrad & .1 & 5 &1000    \\
        SGD&.5  & 5 & 1000   \\
        
        \bottomrule
    \end{tabular}
\end{table}

\begin{table}[h]
\caption{Hyperparameters that generate highest test accuracy for each optimizer on Reddit}\label{tab:reddit}%
\begin{tabular}{@{}llll@{}}
        \toprule
        Optimizer &Learning Rate &Sampled Neighbors & Batch Size
       \\
        \midrule
        Adam& .01 &2  &1000    \\
        Heavy-Ball SGD&.05  & 2 & 1000   \\
        AMSGrad&.01  & 2 & 1000   \\
        AdaGrad & .05 & 2 & 1000   \\
        SGD& .05 & 2 &1000    \\
        
        \bottomrule
    \end{tabular}
\end{table}

\end{appendices}

\clearpage
\bibliography{sn-bibliography}% common bib file

\end{document}